\documentclass[a4paper,reqno,12pt]{amsart}
\usepackage[usenames,dvipsnames,svgnames,table]{xcolor}
\usepackage[foot]{amsaddr}

\colorlet{mylinkcolor}{violet}
\colorlet{mycitecolor}{YellowOrange}
\colorlet{myurlcolor}{Aquamarine}
\usepackage{hyperref}
\usepackage{xcolor}
\hypersetup{
    colorlinks,
    linkcolor={red!50!black},
    citecolor={blue!50!black},
    urlcolor={blue!80!black}
}
\usepackage{amssymb, graphicx, enumerate, hyperref, color}
\newtheorem{theorem}{Theorem}[section]

\newtheorem{lemma}[theorem]{Lemma}

\newcommand{\DEF}{\sl}
\newcommand{\N}{\mathbb{N}}
\newcommand{\R}{\mathbb{R}}

\newcommand{\drd}{\mathrm{drd}}

\renewcommand{\leq}{\leqslant}
\renewcommand{\geq}{\geqslant}

\newcommand{\fivefive}{pure degree-$5$}
\newcommand{\fivefour}{mixed degree-$5$}

\newcommand{\verts}[1]{|V(#1)|}
\newcommand{\edges}[1]{|E(#1)|}

\begin{document}

\title[$K_{4}$-Minor-Free Induced Subgraphs of Sparse Connected Graphs]{$K_{4}$-Minor-Free Induced Subgraphs\\ of Sparse Connected Graphs}

\author[G.~Joret]{Gwena\"{e}l Joret}
\address[G.~Joret]{D\'epartement d'Informatique \\
  Universit\'e Libre de Bruxelles\\
  Brussels\\
  Belgium}
\email{gjoret@ulb.ac.be}

\author[D.~R.~Wood]{David R. Wood}
\address[D.~R.~Wood]{School of Mathematical Sciences \\
  Monash University \\
  Melbourne\\
  Australia}

\email{david.wood@monash.edu}

\thanks{G.\ Joret was supported by a
DECRA Fellowship from the Australian Research Council during part of the project. D.\ R.\ Wood is supported by the Australian Research Council.}

\date{\today}

\sloppy

\maketitle

\begin{abstract}
We prove that every connected graph $G$ with $m$ edges contains a set $X$ of at most $\frac{3}{16}(m + 1)$ vertices such that $G-X$ has no $K_4$ minor, or equivalently, has treewidth at most $2$. This bound is best possible. 
Connectivity is essential: If $G$ is not connected then only a bound of $\frac{1}{5}m$ can be guaranteed.
\end{abstract}

\section{New Introduction}

We consider the minimum size of  a subset $X$ of vertices
to remove from a graph $G$ to obtain a graph with no $K_{4}$ minor, or equivalently, a graph of treewidth at most $2$. Denote this quantity by $s(G)$.
We are interested in bounding $s(G)$ from above by a function of the number $m$ of edges of $G$.
Several authors independently proved the following upper bound:
\begin{theorem}[Scott and Sorkin~\cite{Scott2003}; Kneis, M\"olle, Richter, and Rossmanith~\cite{Kneis-WG05}; Edwards and Farr~\cite{EF08}; Borradaile, Eppstein, and Zhu~\cite{Eppstein}]
\label{thm:m5}
For every graph $G$ with $m$ edges,
$$s(G) \leq \frac{1}{5}m.$$
\end{theorem}
The ratio $ \frac{1}{5}$ is best possible: If $G$ is the disjoint union of $k$ copies of $K_{5}$ then $G$ has $m=10k$ edges and $s(G)=2k=\frac{1}{5} m$.
While $s(G) \leq \frac{1}{5}m$ is tight for arbitrary graphs, this bound can be improved if we restrict our attention to {\em connected} graphs. The following is our main result. 

\begin{theorem}
\label{th-2-fvs-connected}
For every connected graph $G$ with $m$ edges,
$$s(G) \leq \frac{3}{16}(m + 1).$$
\end{theorem}

This bound is best possible: Take $G$ to be the disjoint union of $k$ copies of $K_{6}$, and then add $k-1$ edges so that the resulting graph is connected. Then $s(G) = 3k$ and $m=16k-1$, and thus $s(G) = \frac{3}{16}(m + 1)$.

Edwards and Farr~\cite{EF12} previously studied $s(G)$ for connected graphs with maximum degree 6. 

\begin{theorem}[Edwards and Farr~\cite{EF12}]
\label{thm:EF}
For every connected graph $G$ with maximum degree at most $6$ and $n_{d}$ vertices of degree $d$,
$$s(G) \leq \frac14 n_{3} + \frac38 n_{4} + \frac{19}{40} n_{5} + \frac{11}{20} n_{6} + \frac{3}{20}.$$
\end{theorem}

The best upper bound in terms of $m$ one could hope to prove from 
Theorem~\ref{thm:EF} is $$s(G) \leq  \frac{19}{100}m + \frac{3}{20}$$
(because of the 5-regular case). And with some work this is possible, even for graphs with unbounded degree. 
Nevertheless, Theorem~\ref{th-2-fvs-connected} improves upon this bound. 
Also note that for connected graphs with maximum degree 4, Theorem~\ref{thm:EF} implies $$s(G) \leq  \frac{3}{16}m + \frac{3}{20},$$ matching the coefficient in Theorem~\ref{th-2-fvs-connected}.  As is evident in our proof, the main difficulty in establishing the bound $\frac{3}{16}(m+1)$ in general is the delicate handling of vertices of degree $5$.


We continue this introduction with an overview of related works, emphasizing the different motivations for studying the invariant $s(G)$ that appeared in the literature.

\subsection*{Max-$2$-CSP}
In one line of research, bounding the invariant $s(G)$ for sparse graphs $G$ was originally motivated by the design of exponential-time algorithms for a class of combinatorial optimization problems called Max-$2$-CSP.
Informally, Max-$2$-CSP is the class of constraint satisfaction problems with at most two variables per clause
(see e.g.~\cite{SS07} for a precise definition).
This includes such problems as Max Cut, Max-$2$-Sat, Max Independent Set, and many others.
Problems in this class are naturally parameterized by the number $r$ of different values a variable can take;
for instance $r=2$ in the case of binary variables, which is the most common case.
Several exact exponential-time algorithms have been developed for Max-$2$-CSP.
When parameterized by the number $m$ of edges of the constraint graph, the best one to date is an $O^*(r^{\frac{13}{75}m + o(m)})$-time algorithm of Scott and Sorkin~\cite{SS07}.
(The notation $O^*(\cdot)$ omits polynomial factors in $m$ and $r$.)
This algorithm is based on treewidth, its key component being a proof that the treewidth of $G$ is at most  $\frac{13}{75}m + o(m)$.

Unfortunately, this exponential-time treewidth-based algorithm also takes exponential space.
This motivated a long line of research focusing on exponential-time but polynomial-space algorithms for Max-$2$-CSP~\cite{E13, EM14,  GS14, GS12, GK14, Kneis-WG05, Kneis-SIDMA09, SS07}.
It is in this context that the invariant $s(G)$ attracted some attention:
It is well known that a graph with no $K_4$ minor can be reduced to the empty graph by iteratively `eliminating' vertices of degree at most $2$ (see Lemma~\ref{lemma-2fvs} in Section~\ref{sec:preliminaries}).
As it turns out, these operations can also be realized on the constraint graph of a Max-$2$-CSP instance, in the sense that a vertex of degree at most $2$ can be eliminated in polynomial time, giving a smaller equivalent instance (see e.g.~\cite{SS07}).
The interest of this observation is that if $X\subseteq V(G)$ is such that $G-X$ has no $K_4$ minor, then one can simply branch on the vertices in $X$ and obtain an $O^*(r^{|X|})$-time algorithm that uses only polynomial space.
This observation motivated Scott and Sorkin~\cite{Scott2003} and independently Kneis {\it et al.}~\cite{Kneis-WG05}\footnote{It should be noted that a better bound than $s(G) \leq \frac{1}{5}m$  bound was claimed
by Kneis~{\it et al.} in their conference paper from 2005~\cite{Kneis-WG05}:  They claimed $s(G)\leq \frac{23}{120}m +1$ for every graph $G$ and  $s(G)\leq \frac{3}{16}m +1$ for every graph $G$ with maximum degree at most $4$. However, both proofs fail to address connectivity issues, and indeed the two statements are false, as seen by taking disjoint copies of $K_5$.
We note that other proofs in~\cite{Kneis-WG05} are problematic and were later corrected by the authors in their journal version~\cite{Kneis-SIDMA09} and the above two results were dropped (also see the Acknowledgements section of~\cite{Kneis-SIDMA09} as well as the discussion of~\cite{Kneis-WG05} by Scott and Sorkin~\cite[Section~1.3]{SS07}). 
However, the authors did not highlight the fact that the two results in~\cite{Kneis-WG05} were wrong, resulting in an unfortunate confusion in the literature.
}
to show that such a set $X$ with $|X| \leq \frac{m}{5}$ always exists (c.f.\ Theorem~\ref{thm:m5}), and that it can be found efficiently.

This approach was subsequently improved by Scott and Sorkin~\cite{SS07}.
Informally, they showed that it is enough to bound the `reduction depth' of the recursion tree.
This leads to the definition of the following graph invariant which is at most $s(G)$ 
(here we follow the terminology of Edwards~\cite{E13}).
First, let $r(G)$ denote the graph obtained from $G$ by eliminating vertices of degree at most $2$ (as in Lemma~\ref{lemma-2fvs}).
It is known that $r(G)$ is uniquely defined~\cite[Theorem~3.8]{Kneis-SIDMA09}.
Then the {\DEF delete-reduction depth $\drd(G)$} of $G$ is the least non-negative integer $k$ for which there is a sequence $G_1=r(G), G_2, \dots, G_{k+1}$ of graphs and a sequence $V_1, \dots, V_{k}$ of vertex sets such that
\begin{itemize}
\item $V_i \subseteq V(G_i)$ and $V_i$ contains at most one vertex from each component of $G_i$, for each $i\in \{1, \dots, k\}$; \\[-1.5ex]
\item $G_i = r(G_{i-1} - V_{i-1})$ for each $i\in \{2, \dots, k+1\}$; \\[-1.5ex]
\item $G_{k+1}$ is empty.
\end{itemize}
Note that if one also required $|V_i| = 1$ in the above definition, then the resulting invariant would be equal to $s(G)$ (we leave the proof to the reader). It follows that  $\drd(G) \leq s(G)$  for every graph $G$.

As noted by Edwards~\cite{E13}, it follows from the work of Scott and Sorkin~\cite{SS07} that Max-$2$-CSP can be solved in time $O^*(r^{k})$ and polynomial space given a sequence $V_{1}, \dots, V_{k}$ of vertex sets witnessing $\drd(G) \leq k$ as in the above definition.
Scott and Sorkin~\cite{SS07}  obtained a $O^*(r^{\frac{19}{100}m})$-time algorithm by showing that  $\drd(G) \leq \frac{19}{100}m + 2$ when $G$ has $m$ edges (with an algorithmic proof).

Note that if $G$ is not connected then, by definition, $\drd(G)$ is  the maximum of  $\drd(H)$ over all components $H$ of $G$.
Thus when bounding  $\drd(G)$, we may assume that $G$ is connected, and hence the upper bound of Scott and Sorkin on $\drd(G)$ can be viewed as a precursor of the bound $s(G) \leq \frac{19}{100}m + \frac{3}{20}$ of Edwards and Farr~\cite{EF12} for connected graphs $G$. 
As noted earlier, our upper bound of $\frac{3}{16}(m+1)$ on $s(G)$ when $G$ is connected is tight.
However, better upper bounds can be achieved for $\drd(G)$:
Edwards~\cite{E13} proved that $\drd(G) \leq \frac{9}{50}m + o(m)$.
Again, that proof is algorithmic, and yields an $O^*(r^{\frac{9}{50}m + o(m)})$-time, polynomial space, algorithm for Max-$2$-CSP, which is the best known to date.
Another algorithm matching this time complexity was described recently by Gaspers and Sorkin~\cite{GS14}.

\subsection*{Graph Drawing}
The invariant $s(G)$ is also among a set of closely related graph invariants studied by the graph drawing community. One general approach to draw a graph $G$ in the plane is to first identify an induced subgraph $H$ of $G$ that can be drawn nicely---e.g.\ a forest or a planar graph---and then add somehow the remaining vertices and edges to the drawing in a second phase. If the graph $G$ is sparse enough, one might hope to be able to embed a large fraction of the graph in the first phase, which motivates extremal questions of the following type: Given a class of graphs $\mathcal{C}$ that is considered to have nice drawings (say, planar graphs), what is the least integer $f(m)$ such that every $n$-vertex $m$-edge graph $G$ has an induced subgraph on at least $n - f(m)$ vertices in $\mathcal{C}$; or equivalently, has a set $X$ of vertices with $|X| \leq f(m)$ such that $G-X$ is in $\mathcal{C}$?

Borradaile, Eppstein, and Zhu~\cite{Eppstein} recently considered these questions with $\mathcal{C}$ being the class of (1) pseudo-forests (at most one cycle in each component); (2) $K_4$-minor-free graphs, and (3) planar graphs.
They proved the following upper bounds on the corresponding function $f(m)$: (1) $f(m) \leq \frac{2}{9}m$; (2) $f(m) \leq \frac{1}{5}m$ (c.f.\ Theorem~\ref{thm:m5}), and (3) $f(m) \leq \frac{23}{120}m$.
We note that their second result---i.e.\  $s(G) \leq \frac{1}{5}m$---was also observed implicitly in an earlier work by Edwards and Farr~\cite{EF08}.\footnote{Although that bound is not stated explicitly in~\cite{EF08}, it is a corollary of Lemma~3 in that paper.}

The result of Borradaile {\it et al.}~\cite{Eppstein} for planar graphs improves on an earlier bound of  Edwards and Farr~\cite{EF08} and is in fact slightly stronger than as stated above: It holds for the class $\mathcal{C}$ of planar graphs with treewidth at most $3$.
Nevertheless, this remains the best known bound for arbitrary planar graphs.
It was suggested by Rossmanith that an upper bound of the form $f(m) \leq \frac{1}{6}m + o(m)$ might hold for planar graphs  (see~\cite{Eppstein}).
Borradaile {\it et al.}~\cite{Eppstein} showed that, if true, this would be best possible: They proved that if $\mathcal{C}$ is any proper minor-closed class then $f(m) \geq \frac{1}{6}m - o(m)$.

Note that, as in the context of Max-$2$-CSP, here one is also naturally interested in computing efficiently the small sets $X$ such that $G-X$ is in  $\mathcal{C}$ guaranteed by these theorems, and indeed the results mentioned above come with efficient algorithms (see~\cite{Eppstein}). \\

This paper is organized as follows. We begin in Section~\ref{sec:preliminaries} with the necessary definitions and basic facts about the invariant $s(G)$. In Section~\ref{sec:tools}, we provide some general lemmas which are useful in our proofs, and in Section~\ref{sec-proof} we prove our main result (Theorem~\ref{th-2-fvs-connected}). 

\section{Definitions and Preliminaries}
\label{sec:preliminaries}
All graphs in this paper are finite, simple, and undirected.
Let $N_G(v)$ denote the set of neighbors of vertex $v$ in the graph $G$, and let $\deg_{G}(v) = |N_G(v)|$ denote its degree.
Let $G /uv$ denote the graph obtained from $G$ by contracting the edge $uv$.
(Since we restrict ourselves to simple graphs, loops and parallel edges resulting from edge contractions are deleted.)
A graph $H$ is a {\DEF minor} of a graph $G$ if $H$
can be obtained from a subgraph of $G$ by contracting edges.

A {\DEF $k$-cutset} of a connected graph $G$ is a subset $X \subseteq V(G)$ with $|X|=k$ such that $G-X$ is not connected.   
A {\DEF separation} of $G$ is a pair $(G_{1}, G_{2})$ of induced subgraphs of $G$ such that $V(G_{1})-V(G_{2})$ and $V(G_{2})-V(G_{1})$ are not empty and there are no edges between these two sets. 
The {\DEF order} of the separation  $(G_{1}, G_{2})$ is $|V(G_{1}) \cap V(G_{2})|$.

A {\DEF stable set} in a graph is a set of pairwise non-adjacent vertices. 
The maximum size of a stable set in a graph $G$ is denoted by $\alpha(G)$.
A graph $H$ is a {\DEF subdivision} of a graph $G$ if $H$ can be obtained from $G$ by replacing edges of $G$ with internally disjoint paths having the same endpoints. 
(The internal vertices of the paths are new vertices of the graph.)

Recall the basic observation that a graph $G$ contains $K_4$ as a minor if and only if $G$ contains a $K_4$ subdivision as a subgraph. Graphs with no $K_4$ minors can be characterized in different ways. For instance, these graphs are exactly the graphs of treewidth at most $2$.
The following characterization is useful for our purpose: $G$ has no $K_4$ minor if and only if $G$ can be reduced to the empty graph by iteratively applying the following three reductions (in any order):
\begin{itemize}
\item removing a vertex of degree at most $1$,
\item removing a vertex of degree $2$ lying in a triangle, and
\item contracting an edge incident to a vertex of degree $2$ not in a triangle.
\end{itemize}
Observe in particular that every graph with no $K_4$ minor has a vertex of degree at most $2$.

A subset $X$ of vertices of a graph $G$ such that $G-X$ has no $K_{4}$ minor
is called a {\DEF transversal}.
A {\DEF minimum transversal} of $G$ is a transversal of minimum size; its size
is denoted by $s(G)$.
The following lemma summarizes some simple but essential properties of the invariant $s(G)$, which  imply
the aforementioned characterization of graphs without $K_4$ minors. A proof is included for completeness.

\begin{lemma}
\label{lemma-2fvs}
Suppose that $H$ is a graph obtained from a graph $G$ using one of the following three operations:
\begin{itemize}
\item removing a vertex of degree at most $1$,
\item removing a vertex of degree $2$ lying in a triangle, and
\item contracting an edge incident to a degree-$2$ vertex not in a triangle.
\end{itemize}
Then $s(H) = s(G)$.
\end{lemma}
\begin{proof}
First suppose that $H = G - v$ where $v$ is a vertex of degree at most $1$ in $G$.
Clearly, $s(H) \leq s(G)$. Let $S$ be a minimum transversal of $H$. If $S$ is not a transversal of $G$, then
$G-S$ contains a $K_4$ subdivision, and that subgraph includes $v$. However this is not possible since
$v$ has degree at most $1$. Hence $S$ is a transversal of $G$, and $s(H) \geq s(G)$, implying $s(H) = s(G)$.

Next suppose that $H = G - v$ where $v$ is a vertex of degree $2$ in $G$ lying in a triangle.
Clearly, $s(H) \leq s(G)$ again. Let $S$ be a minimum transversal of $H$. If $S$ is not a transversal of $G$, then
$G-S$ contains a $K_4$ subdivision, and that subgraph includes $v$. However, one can shorten the
$K_4$ subdivision by removing $v$ and adding the edge between its two neighbors (which exists in $G$),
which is then a subgraph of $H-S$, a contradiction. Hence $S$ is a transversal of $G$, and $s(H) \geq s(G)$, implying $s(H) = s(G)$.

Finally, suppose that $H = G / uv$ where $v$ is a vertex of degree $2$ in $G$ not in a triangle, and let
$w$ denote the vertex resulting from the contraction of the edge $uv$.
To show $s(H) \leq s(G)$,  consider a minimum transversal $S$ of $G$. If $u \in S$ or $v \in S$,
then $(S - \{u,v\}) \cup \{w\}$ is a transversal of $H$ of size at most $|S| = s(G)$.
If, on the other hand, $u\notin S$ and $v \notin S$, then $H-S$ is a minor of $G-S$ (obtained by contracting the edge $uv$), and thus does not contain a $K_4$ minor. Hence $S$ is a transversal of $H$.

It remains to show $s(H) \geq s(G)$. Let $S$ be a minimum transversal of $H$. If $w \in S$, then
$T:=(S - w) \cup \{v\}$ is transversal of $G$, of size $|S| = s(G)$. Indeed, $v$ has degree at most $1$ in $G-T$ and thus is not contained in a $K_4$ subdivision in $G-T$, implying that every $K_4$ subdivision
in $G-T$ is also a subgraph of $H-S$.

If $w \notin S$, then we claim that $S$  itself is a transversal of $G$. For suppose not, and let $J$ be
subdivision of $K_4$ in $G -S$. Then $J$ includes $v$, since otherwise $H$ would exist in $J-S$
(replacing vertex $u$ by $w$). But then $v$ has degree $2$ in $J$, and its two neighbors are not adjacent,
implying that $J / uv$ is a $K_4$ subdivision in $H -S$, a contradiction.

Therefore, $s(H) \geq s(G)$, implying $s(H) = s(G)$.
\end{proof}

Note that in the statement of Lemma~\ref{lemma-2fvs} we could have merged the second and third operations simply by writing ``contracting an edge incident to a degree-$2$ vertex'' (since we only consider simple graphs).
We chose however to present it as above to emphasize the two possible situations for a vertex of degree $2$, which will typically be treated separately in the proofs. \\

As mentioned in the introduction, $s(G) \leq  \frac{1}{5}m$ 
for every graph $G$ with $m$ edges,
This can be proved easily using `potential functions'.
As these play an important role in this paper, we provide a quick sketch of the proof.
(This is also a good warm-up for the more technical proofs to come.)
The main idea is to show that there exists a function $\phi:\N \to \R$ (called a {\em potential function}) satisfying the following two properties:
\begin{itemize}
\item $\phi(d) \leq \frac{1}{2}d$ \quad for all $d \geq 0$, and \\[-2ex]
\item $s(G) \leq \frac15 \sum_{v\in V(G)} \phi(\deg(v))$ \quad for every graph $G$.
\end{itemize}
Given such a function $\phi$, the bound follows then by observing that
 $$s(G)  \leq \frac15 \sum_{v\in V(G)} \phi(\deg(v))  \leq \frac15 \sum_{v\in V(G)} \frac{ \deg(v)}{2} = \frac{1}{5}m.$$
One possible choice for the potential function $\phi$ is the following:
$$
\phi(d) := \left\{
\begin{array}{ll}
0 & \textrm{ if } d\in \{0,1,2\}; \\[1ex]
\frac54 & \textrm{ if } d = 3; \\[1ex]
\frac{1}{2}d & \textrm{ if } d \geq 4.
\end{array}
\right.
$$
It is clear that $\phi(d) \leq \frac{1}{2}d$ for all $d \geq 0$, and it thus remains to show that $s(G) \leq \frac15 \sum_{v\in V(G)} \phi(\deg(v))$ for every graph $G$, which we prove by induction on $|V(G)|$.

First, note that vertices of degree at most $2$ can be eliminated from $G$ without changing $s(G)$: If $v$ has degree at most $2$ then $s(G)=s(G')$, where $G'$ is obtained by first removing $v$ from $G$, and then adding an edge between the two neighbors of $v$ if they were not already adjacent in $G$ (see Lemma~\ref{lemma-2fvs}). Since the potential associated to each vertex of $G'$ is at most its potential
in $G$, we are done by induction in this case.
Thus we may assume that $G$ has no vertex of degree at most $2$.

Now consider a vertex $v$ of maximum degree.
If $\deg(v)=5$, then the potential $\phi(\deg(v))$ of $v$ is at least $\frac52$.
Furthermore, if we remove $v$ from $G$ then the potential of each neighbor of $v$ drops by at least $\frac12$ (since these neighbors have degree at least $3$). 
Thus, $$\sum_{v\in V(G)} \phi(\deg_G(v)) \geq \frac52 + 5 \frac12 + \sum_{v\in V(G')} \phi(\deg_{G'}(v)).$$
Since $s(G) \leq s(G-v) + 1$ the result follows by induction.

Now, if $\deg(v)=4$ then $v$ has potential $2$ and each of its neighbors has degree $3$ or $4$. Thus their potential drops by at least $\frac34$ when removing $v$.
(Here we see why it is useful to set $\phi(3):=\frac54$ instead of simply $\frac32$.)
Hence, the sum of potentials decreases by at least $2 + 4\frac34 = 5$ when removing $v$ from $G$, and we are done by induction as before.

Finally, if $v$ has degree $3$, then so do its three neighbors, and the drop
in potential when removing $v$ is $\frac54 + 3\frac54=5$. Once again, we are done by induction. 
This concludes the proof.

As can be seen in this proof, the main benefit of potential functions is that they allow a stronger bound for low-degree vertices, which then helps the induction go through.

\section{Tools}
\label{sec:tools}

In this section we introduce a few lemmas that are used in our main proof.

First we consider graphs $G$ that are edge critical with respect to the invariant $s(G)$:
An edge $e$ of a graph $G$ is said to be {\DEF critical} if $s(G - e) < s(G)$.
The graph $G$ is {\DEF critical} if all its edges are critical.

\begin{lemma}
\label{lem-include-exclude}
Let $G$ be a critical graph. Then, for every edge $uv \in E(G)$, there exists
$S \subseteq V(G) - \{u,v\}$ such that $S \cup \{u\}$ and $S \cup \{v\}$ are
both minimum transversals of $G$.
\end{lemma}
\begin{proof}
Let $S$ be a minimum transversal of $G - uv$. 
Then  $|S| = s(G - uv) = s(G) - 1$, since $uv$ is critical in $G$.
Hence, $u\notin S$ and $v\notin S$ (otherwise, $S$ would be a transversal of $G$).
The graph $G - (S\cup\{u\})$ is a subgraph of $(G-uv) - S$. Thus,
$S\cup\{u\}$ is a transversal of $G$, with size $|S|+1 = s(G)$.
By symmetry, the same is true for $S\cup\{v\}$.
\end{proof}

\begin{lemma}
\label{lem-cutvertex}
Let $G$ be a connected critical graph with $\verts{G} \geq 2$.
Then $G$ is $2$-connected.
\end{lemma}
\begin{proof}
Arguing by contradiction, suppose that $G$ is not $2$-connected. Since $K_{2}$
is not critical, this implies $\verts{G} \geq 3$ and that $G$ has a cutvertex $u$.
Let $G_{1}$, $G_{2}$ be two subgraphs of $G$ such that $G = G_{1} \cup G_{2}$
with $\verts{G_{1}}, \verts{G_{2}} \geq 2$
and $V(G_{1}) \cap V(G_{2}) = \{u\}$.
Let $V_{i} := V(G_{i})$, and let
$v_{i}$ be an arbitrary neighbor of $u$ in $G_{i}$, for $i=1,2$.

Using Lemma~\ref{lem-include-exclude}, let $S_{i}$ ($i=1,2$) be a subset of $V(G) - \{u, v_{i}\}$
such that $S_{i} \cup \{u\}$ and $S_{i} \cup \{v_{i}\}$ are both minimum transversals of $G$.
Clearly, the set
$$
(S_{1} \cap V_{1}) \cup (S_{2} \cap V_{2}) \cup \{u\}
$$
is a transversal of $G$. Since its size is at least
$s(G)
= |S_{2} \cap V_{1}| + |S_{2} \cap V_{2}| + 1$, it follows
$|S_{1} \cap V_{1}| \geq |S_{2} \cap V_{1}|$.
Similarly, since
$$
(S_{2} \cap V_{1}) \cup (S_{1} \cap V_{2}) \cup \{u\}
$$
is also a transversal of $G$, we deduce
$|S_{2} \cap V_{1}| \geq |S_{1} \cap V_{1}|$, and thus
\begin{equation}
\label{eq-S1-S2}
|S_{1} \cap V_{1}| = |S_{2} \cap V_{1}|.
\end{equation}
If $S_2 \cap V_1$ is not a transversal of $G_1$, then there exists a $K_4$ subdivision
$J$ in $G_1 - S_2$. But then $J \subseteq G- (S_2 \cup \{v_2\})$, contradicting the fact that
$S_2 \cup \{v_2\}$ is a transversal of $G$. Thus $S_2 \cap V_1$ is a transversal of $G_1$.
By a symmetric argument, $S_1 \cap V_2$ is a transversal of $G_2$.
It follows that the set $T:=(S_{2} \cap V_{1}) \cup (S_{1} \cap V_{2})$ is
a transversal of $G$. (Note that a $K_4$ subdivision in $G$ cannot contain
both a vertex from $G_1 - u$ and another from $G_2 - u$.)
By~\eqref{eq-S1-S2}, 
$$
|T| = |S_{2} \cap V_{1}| + |S_{1} \cap V_{2}| =
|S_{1} \cap V_{1}| + |S_{1} \cap V_{2}| = s(G) -1,
$$
a contradiction. The lemma follows.
\end{proof}

Now we turn our attention to cycles of even length that are (almost) induced.
A cycle is said to be {\DEF even} or {\DEF odd} according to the parity of its length.
An even cycle $C$ of a graph $G$ is said to be {\DEF almost induced} if either $C$ has no chord, or
$C$ has exactly one chord $e$ and $e$ ``splits'' the cycle into two odd cycles.
Observe that, if $C$ is such a  cycle, then there is a subset $S \subset V(C)$ with
$|S|=\frac{|C|}{2}$ such that $S$ is a stable set of $G$.

\begin{lemma}
\label{lem-cycle}
Let $G$ be a graph with minimum degree at least $3$ and with no subgraph isomorphic to $K_4$.
Then $G$ contains an even cycle. Moreover, every shortest even cycle of $G$
is almost induced.
\end{lemma}
\begin{proof}
It is well known that every graph with minimum degree at least $3$ contains a $K_4$ minor, and thus
a subdivision of $K_{4}$.
Let $H$ be a subgraph of $G$ isomorphic to a subdivision of $K_{4}$.
Consider two distinct vertices $u,v$ of $H$ with degree $3$ in $H$. Since there are three
internally disjoint $uv$-paths
in $H$, two of them have the same parity. Thus the union of these two paths
is an even cycle in $G$.

Let $C$ be any shortest even cycle of $G$. It remains to show that $C$ is almost induced.
This is obviously true if $C$ is induced. If $C$ has exactly one chord $e$ in $G$, then
the two cycles obtained from $C$ using $e$ are odd, since otherwise there would be
a shorter even cycle in $G$. Thus $C$ is also almost induced in that case.

Enumerate the vertices of $C$ in order as $C=v_{1}v_{2}\dots v_{p}$, and
suppose $C$ has two distinct chords $e=v_{i}v_{j}$ ($i < j$) and
$f=v_{k}v_{\ell}$ ($k < \ell$). Then $|C| \geq 6$, since otherwise the vertices of $C$
would induce a subgraph isomorphic to $K_{4}$ in $G$.
By changing the cyclic ordering of $C$ if necessary, we may assume
without loss of generality $i\leq k < j$.
Also, exchanging $e$ and $f$ if necessary, we suppose that $\ell < j$ if $i=k$.
By the observation above, each chord splits $C$ into two odd cycles, that is,
$j - i + 1$ and $\ell - k +1$ are both odd. We distinguish two cases, depending
whether the chords are crossing or not.

First assume $\ell \leq j$, that is, the two chords are not crossing. Then the cycle
$v_{i}v_{i+1}\dots v_{k}v_{\ell}v_{\ell+1}\dots v_{j}$ has length
$(k-i +1) + (j - \ell +1) = (j - i) - (\ell - k) + 2$, which is even. Since the latter
cycle is shorter than $C$, this is a contradiction.

Now suppose $\ell > j$. Then, by assumption, we also have $k > i$.
Recall also that $j > k$, hence $v_{i}, v_{j}, v_{k}$ and $v_{\ell}$ are pairwise distinct,
and the two chords cross.
We can find two other cycles of $G$ using $C$ and these two edges, namely
$C_{1}=v_{i}v_{i+1}\dots v_{k}v_{\ell}v_{\ell - 1}\dots v_{j}$ and
$C_{2}=v_{j}v_{j-1}\dots v_{k}v_{\ell}v_{\ell + 1}\dots v_pv_1 \dots v_{i}$.
Note that $j-k$ and $k-i$ have the same parity since $j-i$ is even, and similarly
that $\ell -j$ and $j-k$ have the same parity since $\ell -k$ is even.
Thus $C_1$ is an even cycle. A symmetric argument shows that $C_2$ is even as well.
Since $|C_{1}| + |C_{2}| = |C| + 4$ and $|C| \geq 6$,
one of these cycles is shorter than $C$, again a contradiction.

It follows that $C$ is almost induced, as claimed.
\end{proof}

\section{Proof of Main Result}
\label{sec-proof}

In order to prove Theorem~\ref{th-2-fvs-connected}, we use the following
potential function: Let $\phi: \N \to \R_{+}$ be defined as
$$
\phi(d) := \left\{
\begin{array}{ll}
0 & \textrm{if } d \in \{0,1,2\}; \\[1ex]
\frac43 & \textrm{if } d=3; \\[1ex]
\frac12 d & \textrm{otherwise}.
\end{array}
\right.
$$
For a graph $G$, define the {\DEF potential} of $G$ as
$$
\phi(G):= \sum_{v\in V(G)} \phi(\deg(v)).
$$
We often simply write $\phi_{G}(u)$ to denote the potential of vertex $u$ in $G$, that is, $\phi\left(\deg_{G}(u)\right)$. 
Also, we use the convention that $\phi_{G}(v)=0$ whenever $v\notin V(G)$. 

Observe that $\phi(G) \leq m$ for every graph $G$ with $m$ edges,
since $\phi(d) \leq d /2$ for all $d \geq 0$. Thus
the bound $s(G) \leq \frac{3}{16}m + \frac{3}{16}$ when $G$ is connected
follows from the following technical theorem. Here  a graph $G$ is said to be {\DEF reduced} if $G$ is $2$-connected, critical, and $\deg(v) \in \{3,4,5\}$ for all $v\in V(G)$.

\begin{theorem} 
\label{prop-phi-2-fvs-connected}
For every connected graph $G$,
\begin{enumerate}[(a)]
\item \label{prop-phi-2-fvs-connected-1st-part}
$\phi(G) \geq \frac{16}{3} s(G) - 1$, and \\[-1.5ex]
\item \label{prop-phi-2-fvs-connected-2nd-part}
$\phi(G) \geq \frac{16}{3} s(G)$ if $G$ is reduced but {\em not} $3$-connected.
\end{enumerate}
\end{theorem}

The remainder of this section is devoted to the proof of Theorem~\ref{prop-phi-2-fvs-connected}. 
Arguing by contradiction, we let $G$ be a counter-example to Theorem~\ref{prop-phi-2-fvs-connected} with $\verts{G} + \edges{G}$ minimum. Our proof is in two steps: First, we prove in Section~\ref{sec-2-fvs-connected-part-b} that $G$ is not a counter-example to part (\ref{prop-phi-2-fvs-connected-2nd-part}) of Theorem~\ref{prop-phi-2-fvs-connected}. 
Then, building on this result, we show in Section~\ref{sec-2-fvs-connected-part-a} that $G$ is not a counter-example to part (\ref{prop-phi-2-fvs-connected-1st-part}) either, giving the desired contradiction. 

We begin with a useful lemma about the potential function $\phi$.

\begin{lemma}
\label{lemma-potential}
Let $a,b,c$ be integers with $a \geq 3$ and $b,c \geq 1$. Then
\begin{enumerate}[(A)]
\item \label{p-a-12} $\phi(a) \geq \phi(a-1) + \frac12$, \\[-1.5ex]
\item \label{p-a-1} $\phi(a) \geq \phi(a-2) + 1$, \\[-1.5ex]
\item \label{p-bc-12} $\phi(b + c -1) \geq \phi(b) + \phi(c) - \frac12$, and \\[-1.5ex]
\item \label{p-bc-1} $\phi(b + c) \geq \phi(b) + \phi(c) + 1$
if moreover $b, c \geq 2$ and $b + c \leq 5$.
\end{enumerate}
\end{lemma}
\begin{proof}
\eqref{p-a-12} and \eqref{p-a-1} are direct consequences of the definition of $\phi$.
\eqref{p-bc-12} is easily checked when $b+c \leq 4$, and for the case $b+c \geq 5$
it follows from the observation that
$$
\phi(b) + \phi(c) - \frac12 \leq \frac{b}{2} + \frac{c}{2} - \frac12 = \phi(b + c -1).
$$
\eqref{p-bc-1} is also straightforward to check.
\end{proof}

\subsection{$G$ Satisfies Part~(\ref{prop-phi-2-fvs-connected-2nd-part}) of Theorem~\ref{prop-phi-2-fvs-connected}}
\label{sec-2-fvs-connected-part-b}

Recall that $G$ is a counter-example to Theorem~\ref{prop-phi-2-fvs-connected} with $\verts{G} + \edges{G}$ minimum. 
In this section we show that $G$ is not a counter-example to part (\ref{prop-phi-2-fvs-connected-2nd-part}) of this theorem.  
This will naturally lead us to consider $2$-cutsets of $G$ in the subsequent proofs. 
Note that given a $2$-cutset $X$, every separation $(G_{1}, G_{2})$ of $G$ with
$V(G_{1}) \cap V(G_{2}) = X$ is such that $G_{1}$ and $G_{2}$ are both connected.

\begin{lemma}
\label{lemma-not-counter-example-b-case-1} 
If $G$ is a counter-example to Theorem~\ref{prop-phi-2-fvs-connected} (\ref{prop-phi-2-fvs-connected-2nd-part}), then every $2$-cutset of $G$ induces an edge.
\end{lemma}
\begin{proof}
Arguing by contradiction, assume that there exists a $2$-cutset
of $G$ that does not induce an edge.
Let $X=\{u,v\}$ be such a cutset and let $(G_{1}, G_{2})$ be a separation of $G$ with
$V(G_{1}) \cap V(G_{2}) = X$.
Since $G$ has minimum degree at least $3$, we have $\verts{G_{1}}, \verts{G_{2}} \geq 4$.
We may assume that $X$ and $(G_{1}, G_{2})$ have been chosen so that
\begin{equation}
\label{eq-no-two-deg-1}
\textrm{at most one vertex in $X$ has degree 1 in $G_{i}$, for each $i\in \{1,2\}$}.
\end{equation}
Indeed, suppose this is not the case. Then without loss of generality
both $u$ and $v$ have degree $1$ in $G_{1}$.
Let $u'$ and $v'$ be the neighbors of respectively $u$ and $v$ in $G_{1}$.
If $u' = v'$, then $u'$ would be a cutvertex of $G$
(recall that $\verts{G_{1}} \geq 4$); thus, $u' \neq v'$.
It follows that $X':=\{u', v\}$ is a cutset of $G$
such that $u'v \notin E(G)$.
Also, the two graphs $G'_{1} := G[V(G_{1}) - \{u\}]$
and $G'_{2} := G[V(G_{2}) \cup \{u'\}]$ define a corresponding separation $(G'_{1}, G'_{2})$ of $G$
with $V(G'_{1}) \cap V(G'_{2}) = X'$.
Moreover, $u'$ and $v$ have degree at least $2$ in $G'_{1}$ and $G'_{2}$, respectively.
Hence, replacing $X$ with $X'$ and $(G_{1}, G_{2})$ with $(G'_{1}, G'_{2})$,
we see that~\eqref{eq-no-two-deg-1} holds.

Let $S_{i}$ ($i=1,2$) be a minimum transversal of $G_{i}$. The set
$S_{1} \cup S_{2} \cup \{u\}$ is a transversal of $G$, implying
\begin{equation}
\label{eq-f2-G-Gi}
s(G) \leq s(G_{1}) + s(G_{2}) + 1.
\end{equation}
The remainder of the proof is split into two cases, depending on whether
$s(G) \leq s(G_{1}) + s(G_{2})$
or $s(G) = s(G_{1}) + s(G_{2}) + 1$.

\medskip{\bf Case 1.} $s(G) \leq s(G_{1}) + s(G_{2})$: 
By~\eqref{eq-no-two-deg-1}, without loss of generality, if $u$ has degree $1$ in $G_{1}$ or $G_{2}$, then
$u$ has degree $1$ in $G_{1}$, and that if $v$ has degree $1$ in $G_{1}$ or $G_{2}$, then
$v$ has degree $1$ in $G_{2}$.
Let $u_{1}$ be a neighbor of $u$ in $G_{1}$.
Let $v_{2}$ be a neighbor of $v$ in $G_{2}$.
Let $G'_{1}$ be obtained from $G_{1}$ by removing $u$ if $u$ has degree $1$ in $G_{1}$.
Define $G'_{2}$ similarly with respect to $G_{2}$ and $v$.
By our assumption on $u$ and $v$, 
\begin{align*}
\phi(G) - \phi(G'_{1}) - \phi(G'_{2}) =\, &
\phi_{G}(u) - \phi_{G'_{1}}(u) - \phi_{G'_{2}}(u) +
\phi_{G}(u_{1}) - \phi_{G'_{1}}(u_{1}) \\
& +
\phi_{G}(v) - \phi_{G'_{1}}(v) - \phi_{G'_{2}}(v) +
\phi_{G}(v_{2}) - \phi_{G'_{2}}(v_{2}).
\end{align*}
(Recall that $\phi_{H}(x)=0$ if $x \notin V(H)$; thus $\phi_{G'_{1}}(u)=0$ if $u \notin V(G'_{1})$
for instance.)
First, we show that
\begin{equation}
\label{eq-diff-u}
\phi_{G}(u) - \phi_{G'_{1}}(u) - \phi_{G'_{2}}(u) +
\phi_{G}(u_{1}) - \phi_{G'_{1}}(u_{1})
\geq 1.
\end{equation}
If $u$ has degree $1$ in $G_{1}$, then
Lemma~\ref{lemma-potential} \eqref{p-a-12} yields
$$
\phi_{G}(u) - \phi_{G'_{1}}(u) - \phi_{G'_{2}}(u) = \phi_{G}(u) - \phi_{G'_{2}}(u) \geq \frac12
$$
and
$$
\phi_{G}(u_{1}) - \phi_{G'_{1}}(u_{1}) \geq \frac12,
$$
implying \eqref{eq-diff-u}. If, on the other hand, $u$ has degree at least $2$ in $G_{1}$, then
$\phi_{G}(u_{1}) = \phi_{G'_{1}}(u_{1})$, and
$$
\phi_{G}(u) - \phi_{G'_{1}}(u) - \phi_{G'_{2}}(u) \geq 1
$$
by Lemma~\ref{lemma-potential} \eqref{p-bc-1}. Thus, \eqref{eq-diff-u} holds in this case too.

By a symmetric argument, 
\begin{equation}
\label{eq-diff-v}
\phi_{G}(v) - \phi_{G'_{1}}(v) - \phi_{G'_{2}}(v) +
\phi_{G}(v_{2}) - \phi_{G'_{2}}(v_{2})
\geq 1.
\end{equation}
Combining \eqref{eq-diff-u} and \eqref{eq-diff-v}, 
\begin{equation}
\label{eq-diff-global}
\phi(G) - \phi(G'_{1}) - \phi(G'_{2}) \geq 2.
\end{equation}

Now, $G'_{1}$ and $G'_{2}$ are obviously connected.
Moreover, $s(G'_{i}) = s(G_{i})$ for $i\in \{1,2\}$ by Lemma~\ref{lemma-2fvs}.
Since both $G'_{1}$ and $G'_{2}$ are smaller than $G$, each of them
satisfies Theorem~\ref{prop-phi-2-fvs-connected}.
Using part (\ref{prop-phi-2-fvs-connected-1st-part}) of the latter theorem
on these two graphs in combination with \eqref{eq-diff-global}, 
$$
\phi(G) \geq
\phi(G'_{1}) + \phi(G'_{2}) + 2 \geq
\left(\frac{16}{3}s(G_{1}) - 1\right) + \left(\frac{16}{3}s(G_{2}) -1\right) + 2
\geq \frac{16}{3}s(G),
$$
which contradicts the fact that $G$ is a counter-example.

\medskip{\bf Case 2.} $s(G) = s(G_{1}) + s(G_{2}) + 1$:
Let $G'_{i} := G_{i} + uv$ for $i=1,2$.
We claim that $s(G'_{i}) > s(G_{i})$ for some $i\in \{1,2\}$.
Indeed, assume otherwise, let $S_{i}$ ($i=1,2$) be a minimum transversal of $G'_{i}$, and
let $S:=S_{1} \cup S_{2}$. Then $G - S$ contains a subgraph $H$ isomorphic to
a subdivision of $K_{4}$ since $|S| \leq s(G_{1}) + s(G_{2}) < s(G)$. Now,
thanks to the edge $uv$, one can easily find from $H$ a subdivision of $K_{4}$ in one of
$G'_{1} - S_{1}$ and $G'_{2} - S_{2}$, a contradiction.
Hence, $s(G'_{i}) > s(G_{i})$  for some $i\in \{1,2\}$; 
without loss of generality, $s(G'_{2}) > s(G_{2})$.

Let $S_{i}$ ($i=1,2$) be a minimum transversal of $G_{i}$. The sets $S_{1}$ and $S_{2}$
cannot include any vertex of the cutset $X$. Indeed, otherwise $S_{1} \cup S_{2}$
would be a transversal of $G$ of size $|S_{1} \cup S_{2}| \leq s(G_{1}) + s(G_{2}) < s(G)$,
a contradiction.
(Here we use that every $K_{4}$ subdivision in $G$ that meets both $V(G_{1})-X$ and $V(G_{2})-X$
 contains both vertices in $X$.)
Now, if $s(G_{1}-u) < s(G_{1})$, then any minimum transversal of $G_{1}-u$ can be extended to one
of $G$ of size at most $s(G)$ by adding $u$, which is not possible as we have seen. Since the same
holds for $G_{1} - v_{1}$, 
\begin{equation}
\label{eq-f2-G1}
s(G_{1} - u) = s(G_{1} - v) = s(G_{1}).
\end{equation}
(Note that this is also true for $G_{2}$ but we will not need this fact.)
Exchanging $u$ and $v$ if necessary, we may assume
\begin{equation}
\label{eq-deg-u-deg-v}
\deg_{G_{1}}(u) \geq \deg_{G_{1}}(v).
\end{equation}

Let $H_{1}$ be the graph $G_{1} - u$ where the vertex $v$ is removed if
$v$ has degree 1 in $G_{1}$.
Note that the graph $G_{1} - u$ is connected (otherwise, $u$ would be a cutvertex of $G$).
Thus, $H_{1}$ is connected as well. Also,
$$
s(H_{1}) = s(G_{1} - u) = s(G_{1}),
$$
by Lemma~\ref{lemma-2fvs} and~\eqref{eq-f2-G1}. Since both $H_{1}$ and $G'_{2}$
are connected and smaller than $G$, they
satisfy Theorem~\ref{prop-phi-2-fvs-connected}.
It follows
$$
\phi(H_{1}) + \phi(G'_{2}) \geq
\left(\frac{16}{3}s(G_{1}) - 1\right) + \left(\frac{16}{3}\left(s(G_{2}) + 1\right) - 1\right)
= \frac{16}{3} s(G) - 2.
$$
(Recall that $s(G'_{2}) > s(G_{2})$ and $s(G) = s(G_{1}) + s(G_{2}) + 1$.)
In order to conclude the proof it is enough to show that 
$\phi(G) \geq \phi(H_{1}) + \phi(G'_{2}) + 2$, since this implies
$\phi(G) \geq \frac{16}{3} s(G)$ and thus that $G$ is not a counter-example.

\medskip{\bf Case 2a.} $v$ has degree 1 in $G_{1}$:
Then $H_{1} = G_{1} - \{u,v\}$.
Also, by~\eqref{eq-no-two-deg-1}, the vertex $u$ has degree at least 2 in $G_{1}$.
Using Lemma~\ref{lemma-potential} \eqref{p-a-12} and \eqref{p-a-1}, and considering the neighbors
of $u$ and $v$ in $G_{1}$ we obtain
$$
\phi(G_{1}) \geq \phi(H_{1}) + \phi_{G_{1}}(u) + \phi_{G_{1}}(v) + \frac32.
$$
(Note that there are two cases to consider: either $u$ and $v$ have one common neighbor or none.)
Also,
$$
\phi_{G}(u) \geq \phi_{G'_{2}}(u) + \frac12,
$$
since $\deg_{G_{1}}(u) \geq 1$ and thus
$\deg_{G'_{2}}(u) \leq \deg_{G}(u) - 1$. Using these observations and the fact that
$\phi_{G}(v) = \phi_{G'_{2}}(v)$,
\begin{align*}
\phi(G) 
& = \Big(\phi(G_{1}) - \phi_{G_{1}}(u) - \phi_{G_{1}}(v) \Big)
+ \phi(G'_{2})
+ \left(\phi_{G}(u) - \phi_{G'_{2}}(u)\right)
+ \left(\phi_{G}(v) - \phi_{G'_{2}}(v)\right)\\
& \geq \left(\phi(H_{1}) +\frac32\right) + \phi(G'_{2}) + \frac12
= \phi(H_{1}) + \phi(G'_{2}) + 2,
\end{align*}
as desired.

\medskip{\bf Case 2b.} $v$ has degree at least $2$ in $G_{1}$:
Then $H_{1} = G_{1} - u$, and $u$ has degree at least 2 in $G_{1}$
by~\eqref{eq-deg-u-deg-v}. Considering the neighbors of $u$ in $G_{1}$, 
\begin{equation}
\label{eq:G1-H1}
\phi(G_{1}) \geq \phi(H_{1}) + \phi_{G_{1}}(u) + \frac12\deg_{G_{1}}(u),
\end{equation}
from Lemma~\ref{lemma-potential} \eqref{p-a-12}.
Similarly,
\begin{equation}
\label{eq:G'2-u}
\phi_{G}(u) \geq \left\{
\begin{array}{lll}
\phi_{G'_{2}}(u) + \frac12 & & \textrm {if } \deg_{G_{1}}(u)=2, \\[1ex]
\phi_{G'_{2}}(u) + 1 & & \textrm {if } \deg_{G_{1}}(u)\geq 3.
\end{array}
\right.
\end{equation}
Also,
\begin{equation}
\label{eq:G1-v}
\phi_{G}(v) \geq \left\{
\begin{array}{lll}
\phi_{G_{1}}(v) + \phi_{G'_{2}}(v) + \frac12 & & \textrm {if } \deg_{G_{1}}(v)=2, \\[1ex]
\phi_{G_{1}}(v) + \phi_{G'_{2}}(v) - \frac12 & & \textrm {if } \deg_{G_{1}}(v)\geq 3.
\end{array}
\right.
\end{equation}
(The bound for the $\deg_{G_{1}}(v)\geq 3$ case is derived
using Lemma~\ref{lemma-potential} \eqref{p-bc-12}.)
By \eqref{eq:G1-H1}, 
\begin{align*}
\phi(G) 
& = \Big(\phi(G_{1}) - \phi_{G_{1}}(u)\Big)
+ \phi(G'_{2})
+ \left(\phi_{G}(u) - \phi_{G'_{2}}(u)\right)
+ \left(\phi_{G}(v) - \phi_{G_{1}}(v) - \phi_{G'_{2}}(v)\right)\\
& \geq \phi(H_{1}) + \frac12\deg_{G_{1}}(u) + \phi(G'_{2})
+ \phi_{G}(u) - \phi_{G'_{2}}(u) + \phi_{G}(v) - \phi_{G_{1}}(v) - \phi_{G'_{2}}(v).
\end{align*}
Combining this with \eqref{eq:G'2-u}, \eqref{eq:G1-v}, and the inequality
$\deg_{G_{1}}(v) \leq \deg_{G_{1}}(u)$ (cf.\ \eqref{eq-deg-u-deg-v}),
we deduce that $\phi(G) \geq \phi(H_{1}) + \phi(G'_{2}) + 2$, both in the case
$\deg_{G_{1}}(v)=2$ and in the case $\deg_{G_{1}}(v)\geq3$.
This concludes the proof.
\end{proof}

\begin{lemma}
\label{lemma-not-counter-example-b-case-2}
$G$ satisfies part (\ref{prop-phi-2-fvs-connected-2nd-part})
of Theorem~\ref{prop-phi-2-fvs-connected}.
\end{lemma}
\begin{proof}
Assume that the lemma is false. Then, by Lemma~\ref{lemma-not-counter-example-b-case-1},
$G$ has a cutset $X=\{u,v\}$ with $uv \in E(G)$.
Let $(G_{1}, G_{2})$ be a separation of $G$ with $V(G_{1}) \cap V(G_{2}) = X$. Notice that
$\verts{G_{i}} \geq 4$ for $i=1,2$ (as in Lemma~\ref{lemma-not-counter-example-b-case-1}).

Let $S_{i}$ ($i=1,2$) be a minimum transversal of $G_{i}$, and
let $S:=S_{1} \cup S_{2}$. If $G - S$ contains
a subdivision of $K_{4}$ then one can find a subdivision of $K_{4}$ in one of
$G_{1} - S_{1}$ and $G_{2} - S_{2}$ thanks to the existence of the edge $uv$.
Thus, $S$ is a transversal of $G$. This implies
$$
s(G) \leq s(G_{1}) + s(G_{2}).
$$

First, suppose $s(G) < s(G_{1}) + s(G_{2})$. Since
the graph $G_{i}$ ($i=1,2$) is connected and smaller than $G$,
Theorem~\ref{prop-phi-2-fvs-connected} gives
$$
\phi(G_{1}) + \phi(G_{2}) \geq
\frac{16}{3} \big(s(G_{1}) + s(G_{2})\big) - 2
\geq \frac{16}{3} s(G) + \frac{16}{3} - 2.
$$
Also,
$$
\phi_{G}(x) \geq \phi_{G_{1}}(x) + \phi_{G_{2}}(x) - \frac12
$$
for each $x\in X$ by Lemma~\ref{lemma-potential} \eqref{p-bc-12}. It follows
\begin{align*}
\phi(G) & =
\phi(G_{1}) + \phi(G_{2}) + \Big(\phi_{G}(u) - \phi_{G_{1}}(u) - \phi_{G_{2}}(u)\Big)
+ \Big(\phi_{G}(v) - \phi_{G_{1}}(v) - \phi_{G_{2}}(v)\Big)\\
& \geq \left(\frac{16}{3} s(G) + \frac{16}{3} - 2\right) - \frac12 - \frac12
\geq \frac{16}{3} s(G),
\end{align*}
contradicting the fact that $G$ is a counter-example to
Theorem~\ref{prop-phi-2-fvs-connected} (\ref{prop-phi-2-fvs-connected-2nd-part}).
Therefore, $s(G) = s(G_{1}) + s(G_{2})$.

If the edge $uv$ is critical in both $G_{1}$ and $G_{2}$, then $S_{i} \cup \{u\}$ ($i=1,2$)
is a transversal of $G_{i}$, where $S_{i}$ ($i=1,2$) is a minimum
transversal of $G_{i} - uv$ of size $s(G_i)-1$. But then $S_{1} \cup S_{2} \cup \{u\}$
is a transversal of $G$ of size $s(G) - 1$, a contradiction. Hence, $uv$ is
not critical in at least one of $G_{1}, G_{2}$, say $G_{1}$.

Let $G'_{1}$ be the graph $G_{1} - uv$ where each vertex of $X$ of degree 1
in $G_{1} - uv$ is removed. The graph $G'_{1}$ is connected and $s(G'_{1}) = s(G_{1}-uv) = s(G_{1})$.
Also, $G'_{1}$ and $G_{2}$ are both smaller than $G$, and thus
satisfy Theorem~\ref{prop-phi-2-fvs-connected}. Hence,
$$
\phi(G'_{1}) + \phi(G_{2}) \geq
\frac{16}{3} \big(s(G_{1}) + s(G_{2})\big) - 2
= \frac{16}{3} s(G) - 2.
$$
Our aim now is to reach a contradiction by showing
$\phi(G) \geq \phi(G'_{1}) + \phi(G_{2}) + 2$.

Consider the vertex $u$ and let $u_{1}$ be a neighbor of $u$ in $G_{1} -uv$. Then,
by Lemma~\ref{lemma-potential} \eqref{p-a-12},
\begin{equation}
\label{eq:u-u1-G'1}
\phi_{G}(u_{1}) - \phi_{G'_{1}}(u_{1}) \geq \left\{
\begin{array}{lll}
\frac12 & & \textrm{if } \deg_{G_{1} - uv}(u)=1, \\[1ex]
0 & & \textrm{otherwise}. \\
\end{array}
\right.
\end{equation}
Similarly,
\begin{equation}
\label{eq:u-u1-G'1-G2}
\phi_{G}(u) - \phi_{G'_{1}}(u) - \phi_{G_{2}}(u)
\geq \left\{
\begin{array}{lll}
\frac12 & & \textrm{if } \deg_{G_{1} - uv}(u)=1, \\[1ex]
1 & & \textrm{otherwise}. \\
\end{array}
\right.
\end{equation}
(The lower bound for the case $\deg_{G_{1} - uv}(u) > 1$
is derived using Lemma~\ref{lemma-potential} \eqref{p-bc-1}.)

Let $v_{1}$ be a neighbor of $v$ in $G_{1}-uv$. Clearly, \eqref{eq:u-u1-G'1}
and \eqref{eq:u-u1-G'1-G2} are also true
if we replace $u$ and $u_{1}$ with $v$ and $v_{1}$, respectively.
Noticing $u_{1} \neq v_{1}$ if $\deg_{G_{1} - uv}(u)=\deg_{G_{1} - uv}(v)=1$ (otherwise,
$u_{1}$ would be a cutvertex of $G$),
we derive from the previous inequalities that
$\phi(G) \geq \phi(G'_{1}) + \phi(G_{2}) + 2$ in every case;  that is,
when the number of vertices of $X$ having degree 1 in $G_{1} - uv$ is $0$, $1$, and $2$.
\end{proof}

By Lemma~\ref{lemma-not-counter-example-b-case-2}, the graph $G$ is not a counter-example to
part (\ref{prop-phi-2-fvs-connected-2nd-part}) of Theorem~\ref{prop-phi-2-fvs-connected}.
Thus, $G$ is a counter-example to part (\ref{prop-phi-2-fvs-connected-1st-part}) of
the theorem.

\subsection{$G$ Satisfies Part~(\ref{prop-phi-2-fvs-connected-1st-part})  of Theorem~\ref{prop-phi-2-fvs-connected}}
\label{sec-2-fvs-connected-part-a}

Building on the fact that $G$ is not a counter-example to   part~(\ref{prop-phi-2-fvs-connected-2nd-part}) of Theorem~\ref{prop-phi-2-fvs-connected}, we show in this section that $G$ is not a counter-example to part~(\ref{prop-phi-2-fvs-connected-1st-part}) either, giving the desired contradiction. 

Properties \eqref{p-a-12} and \eqref{p-a-1} from Lemma~\ref{lemma-potential} have already been
used several times in the proofs and are used many more times in what follows.
For the sake of brevity, we will take them as granted from now on, no longer making explicit reference to
the lemma.

\begin{lemma}
\label{lem:reduced}
The graph $G$ is reduced.
\end{lemma}
\begin{proof}
First, suppose that $G$  has a vertex $v$ of degree at most $2$, and let
$G'$ be obtained from $G$ by applying the appropriate operation from Lemma~\ref{lemma-2fvs} on $v$.
The graph $G'$ is connected with $s(G') = s(G)$ and $\phi(G') \leq \phi(G)$. This implies
that $G'$ is a smaller counter-example, a contradiction. Hence, $G$ has minimum degree at least $3$.

Next, assume that $G$ has a bridge $uv \in E(G)$, and denote by $G_1$ and $G_2$ the
two components of $G - uv$, with $u\in V(G_{1})$ and $v\in V(G_{2})$.
Since $u$ and $v$ have each degree at least $3$ in $G$, we have
$\phi_{G}(u) \geq \phi_{G_{1}}(u) + \frac12$ and
$\phi_{G}(v) \geq \phi_{G_{2}}(v) + \frac12$. Also, $s(G) = s(G_1) + s(G_2)$, which follows from the fact that no
$K_4$ subdivision in $G$ contains the edge $uv$.
Moreover, $G_{1}$ and $G_{2}$ are both connected and smaller than $G$.
Applying Theorem~\ref{prop-phi-2-fvs-connected} on these two graphs gives
\begin{align*}
\phi(G) & =  \phi(G_1) + \phi(G_2) + \Big(\phi_{G}(u) - \phi_{G_{1}}(u)\Big)
+ \Big(\phi_{G}(v) - \phi_{G_{2}}(v)\Big) \\
& \geq \phi(G_1) + \phi(G_2) + \frac12 + \frac12
\geq \frac{16}{3}\big(s(G_1) + s(G_2)\big) - 2 + 1
= \frac{16}{3} s(G) - 1,
\end{align*}
a contradiction. Hence, $G - e$ is connected for every edge $e\in E(G)$.
This implies that every edge $e \in E(G)$ is critical, since otherwise $G - e$
would be a smaller counter-example. It follows then from Lemma~\ref{lem-cutvertex}
that $G$ is $2$-connected.

Now, suppose that $G$ has a vertex $v$ of degree at least $6$.
Then $G - v$ is connected and $s(G - v) \geq s(G) - 1$.
Moreover, $\phi(G) - \phi(G - v) \geq 3 + 6\cdot\frac12 = 6$, since every
neighbor of $v$ has degree at least $3$ in $G$. It follows
$$
\phi(G) \geq \phi(G - v) + 6
\geq \frac{16}{3} s(G - v) - 1 + 6 \geq \frac{16}{3} s(G ) - 1 -16 / 3 + 6 \geq \frac{16}{3} s(G) - 1,
$$
which is again a contradiction. Hence, $G$ has maximum degree at most $5$, and
therefore $G$ is reduced.
\end{proof}

Vertices of degree $5$ in $G$ are the focus of our attention in the next  lemmas. 
We study the local structure of $G$ around such vertices, eventually concluding that there is no degree-$5$ vertex in $G$.

\begin{lemma}
\label{lemma-no-neighbor-3-at-most-one-4}
If $v\in V(G)$ has degree 5, then $v$ has no neighbor of degree 3, and at most one of degree 4.
\end{lemma}
\begin{proof}
Since $G-v$ is connected and 
$s(G - v) \geq s(G) - 1$, 
$$
\phi(G) = \phi(G - v) + \lambda
\geq \frac{16}{3} s(G - v) - 1 + \lambda \geq \frac{16}{3} s(G ) - 1 -\frac{16}{3} + \lambda,
$$
where
$$
\lambda := \frac52 + \sum_{x \in N_{G}(v)} \Big(\phi_{G}(x) - \phi_{G-v}(x)\Big).
$$
For every integer $d \geq 3$, 
$$
\phi(d) - \phi(d - 1) = \left\{
\begin{array}{ll}
\frac43 & \textrm{ if } d=3; \\[1ex]
\frac23 & \textrm{ if } d=4, \\[1ex]
\frac12 & \textrm{ otherwise}.
\end{array}
\right.
$$
Hence, if $N_{G}(v)$ contains a vertex of degree 3 or two vertices of degree 4, then $\lambda \geq \frac{16}{3}$, implying that $G$ is not a counter-example to Theorem~\ref{prop-phi-2-fvs-connected}(\ref{prop-phi-2-fvs-connected-1st-part}).
\end{proof}

Thus there are only two possibilities for a vertex $v$ of degree 5 in $G$: either all neighbors of $v$ also have degree 5, or four of them have degree 5 and one has degree 4.
In the former case, we call $v$ a {\DEF \fivefive} vertex, while in the latter case we call $v$ a {\DEF \fivefour} vertex.

\begin{lemma}
\label{lemma-3-cutsets}
The graph $G$ is $3$-connected. Moreover, if $X \subseteq V(G)$ is a $3$-cutset of $G$, then $\deg_{G}(v) \in \{3,4\}$ for every vertex $v \in X$.
\end{lemma}
\begin{proof}
We already know that $G$ is reduced by Lemma~\ref{lem:reduced}, and thus in particular $G$ is $2$-connected.
The graph $G$ is also $3$-connected, as otherwise $G$ would also be a counter-example to
Theorem~\ref{prop-phi-2-fvs-connected} (\ref{prop-phi-2-fvs-connected-2nd-part}), contradicting
Lemma~\ref{lemma-not-counter-example-b-case-2}.

Thus it remains to show the second part of the statement. Arguing by contradiction,
suppose that $X \subseteq V(G)$ is a $3$-cutset of $G$ containing a vertex $v$ of degree 5. 
Then $G-v$ is $2$-connected but not $3$-connected.  
Moreover, by Lemma~\ref{lemma-no-neighbor-3-at-most-one-4}, $G - v$ has minimum degree at least $3$.

First, assume that some edge $e\in E(G-v)$ is not critical in $G-v$.
Let $G' := (G-v)-e$. Observe that $\phi(G - v) \geq  \phi(G') + 1$.
Since $s(G') = s(G-v) \geq s(G) -1$ and the graph $G'$ is connected and smaller than $G$,
Theorem~\ref{prop-phi-2-fvs-connected} gives
\begin{align*}
\phi(G) & = \phi(G - v) +  \frac52 + \sum_{x \in N_{G}(v)} \Big(\phi_{G}(x) - \phi_{G-v}(x)\Big)  \\
& \geq \phi(G - v) + 5 \\
& \geq \phi(G') + 6 
\geq \frac{16}{3}s(G') - 1 + 6
\geq \frac{16}{3}s(G) - \frac{16}{3} - 1 + 6
> \frac{16}{3}s(G) - 1,
\end{align*}
a contradiction. Hence, $G-v$ is critical. In particular, $G-v$ is reduced.
Using Theorem~\ref{prop-phi-2-fvs-connected} (\ref{prop-phi-2-fvs-connected-2nd-part})
on $G-v$, 
\begin{align*}
\phi(G) & = \phi(G - v) +  \frac52 + \sum_{x \in N_{G}(v)} \Big(\phi_{G}(x) - \phi_{G-v}(x)\Big)  \\
& \geq \phi(G - v) + 5 \geq \frac{16}{3}s(G - v) + 5
\geq \frac{16}{3}s(G) - \frac{16}{3} + 5 > \frac{16}{3}s(G) - 1,
\end{align*}
which is again a contradiction.
\end{proof}

The above proof shows why it is useful to prove a stronger bound in the case that the graph is reduced but not $3$-connected. 
(Indeed, it is when we were trying to prove Lemma~\ref{lemma-3-cutsets} that we were led to add part (\ref{prop-phi-2-fvs-connected-2nd-part}) to Theorem~\ref{prop-phi-2-fvs-connected}.) 

For $v \in V(G)$, denote by $G_{v}$ the subgraph of $G$ induced by the neighborhood of $v$.
(Note that $v$ is not included in $G_v$.)

\begin{lemma}
\label{lemma-fivefive-no-stable-set-size-3}
Suppose that $v\in V(G)$ is a {\fivefive} vertex.
Then $\alpha(G_{v})=2$.
\end{lemma}
\begin{proof}
First, suppose that $G_{v}$ has a stable set $S$ of size $3$.
By Lemma~\ref{lemma-3-cutsets}, the graph $G-S$ is connected.
Since $S$ is a stable set, each vertex $u\in S$ has exactly four neighbors
outside $S \cup \{v\}$. Let $T$ be the set of all such neighbors.
Note that each vertex in $T$ has degree at least $4$ by Lemma~\ref{lemma-no-neighbor-3-at-most-one-4}, and has
between one and three neighbors in $S$.
Using $s(G - S) \geq s(G) - 3$ and Theorem~\ref{prop-phi-2-fvs-connected}
on $G-S$, 
\begin{align*}
\phi(G) 
& = \phi(G - S) + (\phi_{G}(v) - \phi_{G-S}(v)) +
\sum_{u\in S}\phi_{G}(u) + \sum_{u\in T}(\phi_G(u) - \phi_{G-S}(u))\\
& \geq \phi(G - S) + \frac52 + 3 \frac52 + 6\\
& = \phi(G - S) + 16\\
& \geq \frac{16}{3}s(G - S) - 1 + 16
\geq \frac{16}{3}s(G) - 3\frac{16}{3} - 1 + 16
= \frac{16}{3}s(G) - 1,
\end{align*}
a contradiction. Hence, $\alpha(G_{v}) \leq 2$.

On the other hand, if $\alpha(G_{v})=1$, then $G$ is isomorphic to $K_{6}$,
because $G$ has maximum degree at most $5$. Since $K_{6}$
is not a counter-example to Theorem~\ref{prop-phi-2-fvs-connected}, we deduce
$\alpha(G_{v})=2$, as claimed.
\end{proof}

For a subset $S$ of vertices of $G$, let $e_{G}(S)$ denote the number of edges
between $S$ and $V(G) - S$.  
Observe that, if $|S| \leq 4$ and all vertices in $V(G)-S$ having
a neighbor in $S$ have degree at least $4$, then
$$
\phi(G) \geq \phi(G-S) + \sum_{u\in S}\phi_{G}(u) + \frac12 e_G(S).
$$
This observation is used in the subsequent proofs.

A {\DEF sparse bipartition} of a $5$-vertex graph $H$ is a partition $(X, Y)$ of $V(H)$ such that $|X|=3$, $|Y|=2$, and $\edges{H[X]}=\edges{H[Y]} = 1$.

\begin{lemma}
Suppose that $v\in V(G)$ is a {\fivefive} vertex.
Then $G_{v}$ does not admit a sparse bipartition.
\end{lemma}
\begin{proof}
Arguing by contradiction,
assume $G_{v}$ has a sparse bipartition $(X, Y)$. Let
$S := X \cup \{v\}$. In $G - X$, the two neighbors of $v$ are adjacent, implying
$s(G - S) = s(G - X) \geq s(G) - 3$.
By Lemma~\ref{lemma-3-cutsets},  $G-X$ is connected, and hence  $G-S$ is also connected.
It follows that 
\begin{align*}
\phi(G) \geq \phi(G - S) +
\sum_{u\in S}\phi_{G}(u) + \frac12 e_{G}(S)
& = \phi(G - S) + 10 + 6\\
& \geq \frac{16}{3}s(G - S) - 1 + 16 \geq \frac{16}{3}s(G) - 1,
\end{align*}
a contradiction.
\end{proof}

\begin{lemma}
\label{lemma:C5}
Suppose that $v\in V(G)$ is a {\fivefive} vertex.
Then $G_{v}$ is isomorphic to $C_{5}$, the cycle on five vertices.
\end{lemma}
\begin{proof}
As shown in the previous two lemmas, $\alpha(G_{v})=2$ and $G_{v}$ does not admit a sparse bipartition.
Observe furthermore that the subset $X$ of vertices of $G_{v}$
having degree at most 3 in $G_{v}$
is a cutset of $G$ separating $\{v\} \cup (V(G_v) \setminus X)$ from the rest of the graph.
It follows from Lemma~\ref{lemma-3-cutsets} that $|X| \geq 4$, and thus
that $G_{v}$ contains at most one vertex of degree 4 in $G_v$.

We leave it to the reader to check that, up to isomorphism, there are exactly three graphs $H$ on five vertices with $\alpha(H)=2$, having no sparse bipartition, and with at most one vertex of degree $4$:
the wheel $W_{5}$, the graph $K_{4}$ with an edge subdivided once (which we denote by $W_{5}^{-}$), and the cycle $C_{5}$. 
Thus $G_v$ is isomorphic to one of these three graphs, illustrated in Figure~\ref{fig-three-graphs-Gv}.

\begin{figure}[h]
\centering
\includegraphics[width=0.7\textwidth]{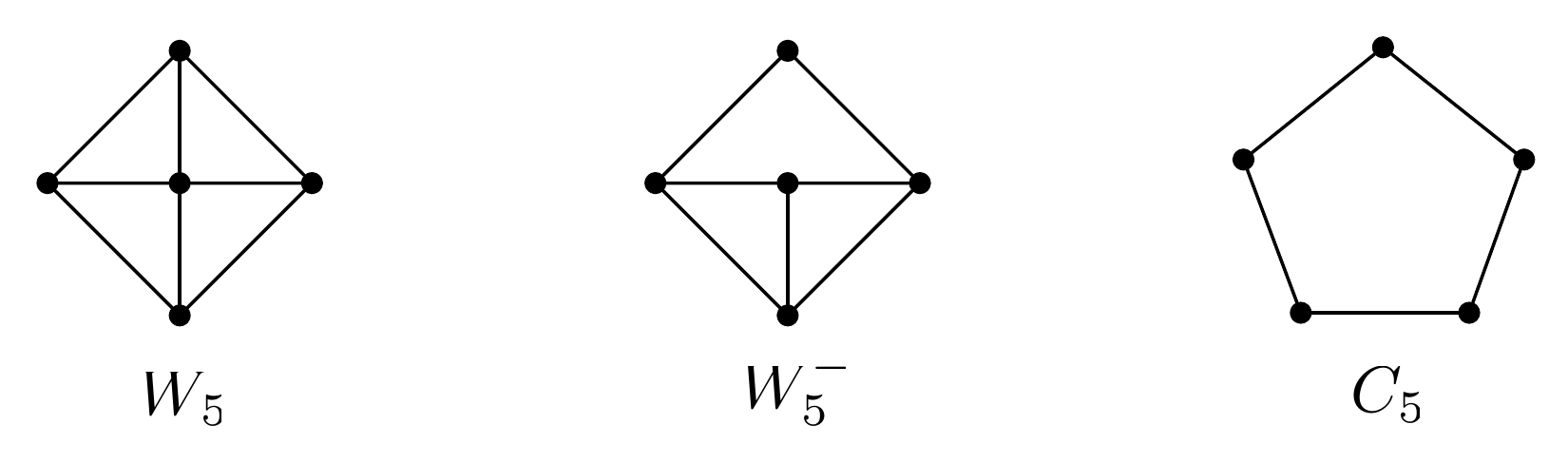}
\caption{\label{fig-three-graphs-Gv}Three possibilities for the graph $G_{v}$.}
\end{figure}

First, suppose $G_{v} \cong W_{5}$, and let $X \subseteq V(G_{v})$ be
such that $|X|=3$ and every vertex in $X$ has degree 3 in $G_{v}$.
Let also $S:= X \cup \{v\}$.
Then $s(G - S) = s(G-X) \geq s(G) - 3$.
Also, $G-X$ is connected by Lemma~\ref{lemma-3-cutsets}, and hence so is $G-S$.
It follows that
\begin{align*}
\phi(G) 
& \geq \phi(G - S) + 4\phi(5) + \big(\phi(5) - \phi(1) \big) + \big(\phi(5) - \phi(2) \big) + 3\frac12\\
& =\phi(G - S) + 16 +\frac12
\geq \frac{16}{3}s(G - S) - 1 + 16 +\frac12
\geq \frac{16}{3}s(G) - 1,
\end{align*}
a contradiction. Hence, $G_{v}$ is not isomorphic to $W_{5}$.

Now, assume $G_{v} \cong W_{5}^{-}$, and let $X \subseteq V(G_{v})$
consist of the unique vertex of $G_{v}$ with degree 2 together with its two neighbors.
Let $S:= X \cup \{v\}$. Similarly as above,
the graph $G-S$ is connected, and we can check that
$s(G - S) \geq s(G) - 3$ and
$\phi(G) \geq \phi(G - S) + 17$, implying again
$\phi(G) \geq \frac{16}{3} s(G) - 1$.
Thus, $G_{v}$ is not isomorphic to $W_{5}^{-}$ either.

It follows that $G_{v}$ is isomorphic to $C_{5}$.
\end{proof}

Next, we deal with the case where $v\in V(G)$ is a {\fivefive} vertex and
$G_{v}$ is isomorphic to $C_{5}$.

\begin{lemma}
There is no {\fivefive} vertex in $G$.
\end{lemma}
\begin{proof}
Suppose that $v\in V(G)$ is a {\fivefive} vertex. Thus, $G_{v} \cong C_{5}$ by Lemma~\ref{lemma:C5}.
First, assume that $G_{v}$ contains at least three vertices which are
{\fivefour} vertices in $G$.
Among them, there are two vertices $x$ and $y$ which are not adjacent (since $G_{v}$ has no triangle).
Moreover, $x$ and $y$ have at least two common neighbors with degree $5$ in $G$, namely
$v$ and some vertex in $V(G_{v}) - \{x,y\}$.
Combining these observations with the fact that
$x$ and $y$ each have a degree-4 neighbor (possibly the same vertex), one obtains
\begin{align*}
\phi(G) 
& \geq \phi(G - \{x,y\}) + 2\phi(5) + 2\big(\phi(5) - \phi(3) \big) + 2\frac23 + 4\frac12\\
& =\phi(G - \{x,y\}) + 10 +\frac23
\geq \frac{16}{3}s(G - \{x,y\}) - 1 + 10 +\frac23
\geq \frac{16}{3}s(G) - 1,
\end{align*}
a contradiction. Hence, there are at most two {\fivefour} vertices in $G_{v}$.

Now, there are two {\fivefive} vertices $x$ and $y$ which are adjacent in $G_{v}$.
Since $G_{v}, G_{x}$, and $G_{y}$ are all
isomorphic to $C_{5}$ by Lemma~\ref{lemma:C5}, the graph $G[Z]$ induced by the set $Z:= V(G_{v}) \cup
V(G_{x}) \cup V(G_{y})$ is as depicted in Figure~\ref{fig-two-adj-five-five}.

Let $S$ be the set consisting of the three vertices in $Z-\{v,x,y\}$ that have
exactly two neighbors in $\{v,x,y\}$.
\begin{figure}
\centering
\includegraphics[width = 0.3\textwidth]{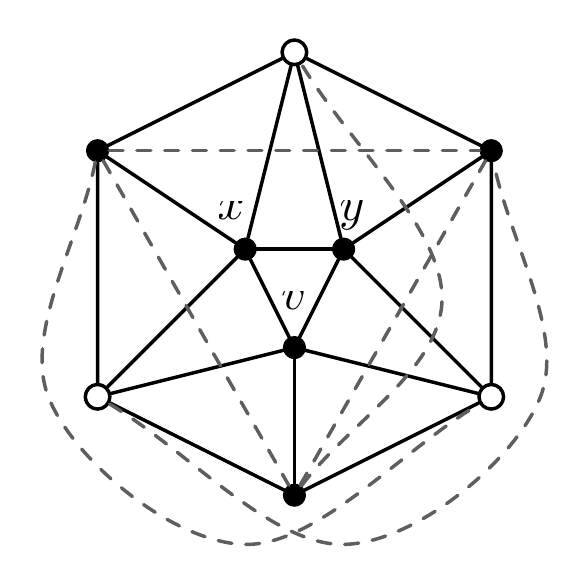}
\caption{\label{fig-two-adj-five-five}The subgraph of $G$ induced by $Z$. Dashed edges indicate edges
that may or may not be present in $G[Z]$; white vertices are the vertices in $S$.}
\end{figure}
The set $S$ is a stable set of $G$ and the graph $G-S$ is connected
(as follows from Lemma~\ref{lemma-3-cutsets}).
Let $a$ and $b$ be the number of vertices with respectively two and three neighbors in $S$.
(Thus there are exactly $a$ edges between $S$ and $V(G) - Z$.)
It follows
\begin{align*}
\phi(G) & \geq \phi(G - S) + 3\phi(5) + (3+a)\big(\phi(5) - \phi(3) \big) + b\big(\phi(5) - \phi(2) \big) + a\frac12 \\
& \geq \phi(G - S) +
3\phi(5) + 6\big(\phi(5) - \phi(3) \big) + 3\frac12 \\
& = \phi(G - S) + 16
\geq \frac{16}{3}s(G - S) - 1 + 16
\geq \frac{16}{3}s(G) - 1,
\end{align*}
which is again a contradiction.
\end{proof}

Our next aim is to prove that $G$ does not have {\fivefour} vertices either.
We follow an approach similar to the one used for {\fivefive} vertices.

A {\DEF diamond} in $G$ is an induced subgraph $H$ of $G$ isomorphic to $K_{4} - e$
(the graph $K_{4}$ minus an edge) with the property
that the two vertices with degree 2 in $H$ both have degree 5 in $G$.

\begin{lemma}
\label{lemma-5-4-diamond}
There is no diamond in $G$.
\end{lemma}
\begin{proof}
Suppose the contrary that $G$ contains a diamond $H$. Let $x, y$ be the two vertices with degree 2
in $H$. If some vertex of $V(H) \setminus \{x,y\}$ has degree 4 in $G$, then 
$$\phi(G) \geq \phi(G - \{x,y\}) +
2\phi(5) + \big(\phi(4) - \phi(2) \big) + \big(\phi(5) - \phi(3) \big) + 6\frac12
\geq \phi(G - \{x,y\}) + 10 +\frac23.$$
If, on the other hand, both vertices in $V(H) - \{x,y\}$ have degree 5 in $G$, then
\begin{align*}
\phi(G) \geq \phi(G - \{x,y\}) +
2\phi(5) + 2\big(\phi(5) - \phi(3) \big) + 2\frac23 + 4\frac12
= \phi(G - \{x,y\}) + 10 +\frac23.
\end{align*}
(Here, we use that $x$ and $y$ each have a degree-$4$ neighbor in $G$.)
Thus, in both cases,
\begin{align*}
\phi(G) \geq \phi(G - \{x,y\}) + 10 +\frac23
\geq \frac{16}{3}s(G - \{x,y\}) - 1 + 10 +\frac23
\geq \frac{16}{3}s(G) - 1,
\end{align*}
contradicting the fact that $G$ is a counter-example.
\end{proof}

\begin{lemma}
\label{lemma-5-4-stable-set}
Suppose $v\in V(G)$ is a {\fivefour} vertex and let $w$ be the unique neighbor
of $v$ of degree 4.
Then $w\in S$ for every stable set $S$ of $G_{v}$ of size $3$.
\end{lemma}
\begin{proof}
If there is a stable set of size $3$ in $G_v$ avoiding $w$, then we can reach a contradiction
exactly as in the first part
of the proof of Lemma~\ref{lemma-fivefive-no-stable-set-size-3}.
\end{proof}

Suppose that $v\in V(G)$ is a {\fivefour} vertex.
Here, we define a sparse bipartition of $G_{v}$ as a partition $(X,Y)$ of $V(G_{v})$
such that $|X|=3, |Y|=2$, $\edges{G[X]} = \edges{G[Y]} = 1$, as for {\fivefive} vertices,
{\em but} with the additional requirement that
$Y$ includes the unique neighbor of $v$ having degree 4 in $G$.

\begin{lemma}
\label{lemma-5-4-sparbip}
Suppose that $v\in V(G)$ is a {\fivefour} vertex.
Then $G_{v}$ does not admit a sparse bipartition.
\end{lemma}
\begin{proof}
Assume $(X, Y)$ is a sparse bipartition of $G_{v}$.
Since the two vertices in $Y$ are adjacent,
$s(G - S) = s(G - X) \geq s(G) - 3$, where $S:= X \cup\{v\}$.
Let $w$ be the vertex in $Y$ of degree $4$. Observe that $w$ sees at most three vertices in $S$, since $w$ is
adjacent to the other vertex in $Y$. Thus every degree-$4$ vertex in $G$ sees at most three vertices from $S$.
Since $G-S$ is connected, 
\begin{align*}
\phi(G) \geq \phi(G - S) +
4\phi(5) + 4\frac23 + 8\frac12
\geq \phi(G - S) + 16
\geq \frac{16}{3}s(G) - 1,
\end{align*}
a contradiction.
\end{proof}

\begin{lemma}
\label{lemma-5-4-at-most-one-neighbor}
Suppose that $v\in V(G)$ is a {\fivefour} vertex and let $w$ be the unique neighbor
of $v$ of degree 4. Then $w$ has degree at most $1$ in $G_{v}$.
\end{lemma}
\begin{proof}
Arguing by contradiction,
assume $x, y$ are two distinct neighbors of $w$ in $G_{v}$.
If $xy \notin E(G)$, then we obtain
\begin{align*}
\phi(G) & \geq \phi(G - \{x,y\}) +
2\phi(5) + \big(\phi(5) - \phi(3) \big) + \big(\phi(4) - \phi(2) \big) + 6\frac12 \\
& \geq\phi(G - \{x,y\}) + 10 +\frac23
\geq \frac{16}{3}s(G - \{x,y\}) - 1 + 10 +\frac23
\geq \frac{16}{3}s(G) - 1,
\end{align*}
a contradiction.
Thus, we may assume $xy \in E(G)$.

Let $Q := N_{G}(v) - \{w\}$.
We have $\alpha(G[Q]) \leq 2$ by Lemma~\ref{lemma-5-4-stable-set}.
Also, if $\edges{G[X]}=2$ for some $X \subseteq Q$ with $|X|=3$, then
the subgraph of $G$ induced by $X \cup \{v\}$ is a diamond of $G$,
which Lemma~\ref{lemma-5-4-diamond} forbids.
On the other hand, if we had $\edges{G[X]}=1$ for some $X \in \left\{ Q - \{x\},
Q - \{y\}\right\}$, then $(X, N_{G}(v) - X)$ would be a sparse bipartition of
$G_{v}$, which would contradict Lemma~\ref{lemma-5-4-sparbip}.

Since $xy \in E(G)$, it follows from the previous observations that $G[Q]$ is a complete graph.
In particular, $x$ and $y$ have no neighbor outside $N_{G}(v) \cup \{v\}$.
Notice also that some vertex in $Q$ is not adjacent to $w$, and thus has a neighbor outside
$N_{G}(v) \cup \{v\}$. It follows that $Z:=N_{G}(v) - \{x,y\}$ separates $\{v, x, y\}$ from the rest
of the graph.
Since $|Z|=3$ and $Z$ includes a vertex of degree 5, this contradicts
Lemma~\ref{lemma-3-cutsets}.
\end{proof}

\begin{lemma}
\label{lemma-5-4-two-graphs}
Suppose that $v\in V(G)$ is a {\fivefour} vertex and let $w$ be the unique neighbor
of $v$ with degree 4.
Then $G_{v}$ is one of the two graphs depicted in Figure~\ref{fig-five-four}.
\end{lemma}

\begin{figure}[h]
\center
\includegraphics[width=0.36\textwidth]{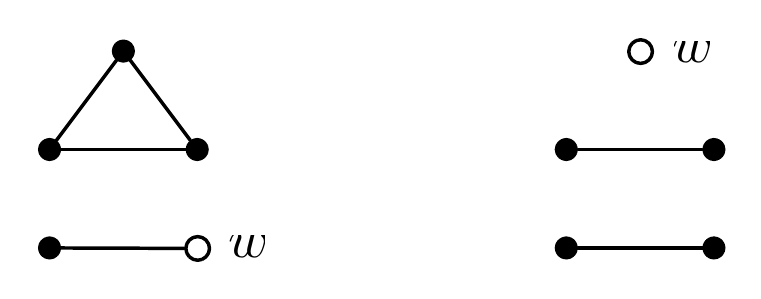}
\caption{\label{fig-five-four}Two possibilities for the graph $G_{v}$
when $v$ is a {\fivefour} vertex.}
\end{figure}

\begin{proof}
Let $Q := N_{G}(v) - \{w\}$.
Lemmas~\ref{lemma-5-4-diamond} and~\ref{lemma-5-4-stable-set} imply that
$G[Q]$ has at most two components and that each component is a complete graph.
Hence, $G[Q]$ is isomorphic to $K_{4}$, $K_{3} \cup K_{1}$, or $K_{2} \cup K_{2}$.
(As expected, $K_t \cup K_{\ell}$ denotes the disjoint union of $K_t$ and $K_{\ell}$.)

First, assume $G[Q] \cong K_{4}$, and let $x,y$ be two distinct vertices of $Q$
such that $wx\notin E(G)$ and $wy\notin E(G)$ (such vertices exist by Lemma~\ref{lemma-5-4-at-most-one-neighbor}).
Let $S:= \{v,w,x,y\}$.
Then $s(G - S) = s(G - \{w,x,y\}) \geq s(G) - 3$. Using
that each of $x$ and $y$ has a neighbor outside $N_{G}(v) \cup \{v\}$,
that $w$ has at least two neighbors outside $N_{G}(v) \cup \{v\}$
(cf.\ Lemma~\ref{lemma-5-4-at-most-one-neighbor}) and, as usual,
that $G-S$ is connected, 
\begin{align*}
\phi(G) & \geq \phi(G - S) +
3\phi(5) + \phi(4) + 2\big(\phi(5) - \phi(2) \big) + 4\frac12\\
& \geq\phi(G - S) + 16
\geq \frac{16}{3}s(G - S) - 1 + 16
\geq \frac{16}{3}s(G) - 1,
\end{align*}
a contradiction. Thus $G[Q]$ is not isomorphic to $K_{4}$.

Recall that $w$ has degree at most $1$ in $G_v$, by Lemma~\ref{lemma-5-4-at-most-one-neighbor}.
Suppose that $w$ has degree 1 in $G_{v}$, and let $z$ be its unique neighbor.
We cannot have $\edges{G[N_{G}(v) - \{w,z\}]}=1$, since otherwise
$(N_{G}(v) - \{w,z\},\{w,z\})$ would be a sparse bipartition of $G_{v}$. It follows
that $G[Q]$ is isomorphic to $K_{3} \cup K_{1}$ and that $z$ is the isolated vertex
of that graph. Hence, $G_{v}$ is isomorphic to the left graph
in Figure~\ref{fig-five-four}.

Now, assume $w$ has degree 0 in $G_{v}$. Suppose $G[Q] \cong K_{3} \cup K_{1}$,
and let $X$ be a stable set of $G_{v}$ with $|X|=3$ (thus $w \in X$).
Let $S := X \cup \{v\}$. Similarly as before, we deduce
\begin{align*}
\phi(G) 
& \geq \phi(G - S) +
3\phi(5) + \phi(4) + 2\big(\phi(5) - \phi(3)\big) + 9\frac12\\
& \geq \phi(G - S) + 16
\geq \frac{16}{3}s(G - S) - 1 + 16
\geq \frac{16}{3}s(G) - 1,
\end{align*}
a contradiction. Thus $G[Q]\cong K_{2} \cup K_{2}$, and $G_{v}$ is
isomorphic to the graph on the right in Figure~\ref{fig-five-four}.
\end{proof}

\begin{lemma}
\label{lemma-cycle-deg-5}
If $G$ has maximum degree $5$, then
there exists an almost induced even cycle $C$ in $G$
such that every vertex in $C$ has degree 5.
\end{lemma}
\begin{proof}
Following Lemma~\ref{lemma-5-4-two-graphs},
a vertex $v$ of $G$ is said to be of {\DEF type 1}
({\DEF type 2}) if $v$ is a {\fivefour} vertex and
$G_{v}$ is isomorphic to the left (resp.\ right) graph in Figure~\ref{fig-five-four}.

Let $H_{i} \subseteq G$ be the subgraph of $G$ induced by the set of vertices of type $i$,
for $i=1,2$. Notice that, if $v \in V(G)$ is of type $i$, then
so are the four neighbors of $v$ that have degree 5 in $G$. Thus $H_{i}$
is either $4$-regular or empty, for $i=1,2$.

If there is a type 2 vertex in $G$, then
$H_{2}$ is $4$-regular and has no subgraph isomorphic to $K_4$.
Using Lemma~\ref{lem-cycle} on $H_{2}$, we obtain a cycle $C$ as desired.
Thus, we may assume that every degree-5 vertex in $G$ is of type 1.
In particular, there is at least one such vertex.

Now, the graph $H_{1}$ is $4$-regular, and every vertex of
$H_{1}$ is contained in exactly one copy of $K_{4}$.
More precisely, all copies of $K_4$ in $H_1$ are pairwise vertex-disjoint, they cover all of $V(H_1)$, and
each copy sends exactly four edges to other copies of $K_4$ in $H_1$.
Observe also that the set of edges of $H_1$ that link two distinct copies of $K_4$ form a perfect matching of $H_1$.

Let $\widetilde H_{1}$ be the multigraph
obtained by contracting each copy of $K_{4}$ into one vertex
(parallel edges between distinct vertices are kept but loops are removed).
It follows from the previous observation that $\widetilde H_{1}$ is $4$-regular.
Let $\widetilde C$ be any induced cycle of $\widetilde H_{1}$
(note that a cycle of length 2 is allowed).
The cycle $\widetilde C$ naturally corresponds to an induced cycle $C$
of $H_{1}$ having length $2|\widetilde C|$. This latter cycle is as desired:
$C$ has even length, is induced (and thus almost induced) in $G$,
and contains only vertices of degree 5.
\end{proof}

\begin{lemma}
\label{lemma-no-deg-5}
$G$ has maximum degree at most 4.
\end{lemma}
\begin{proof}
Arguing by contradiction, we assume that $G$ has maximum degree $5$.
Let $C$ be a cycle of $G$ as in Lemma~\ref{lemma-cycle-deg-5}.
Enumerate the vertices of
$C$ in order as $v_{1}, v_{2}, \ldots, v_{2k}$ so that $S:=\{v_{1}, v_{3}, \dots, v_{2k-1}\}$
is a stable set of $G$. For $i\in \{1, 2, \dots, k\}$,
let $S_{i}:=\{v_{1}, v_{3}, \dots, v_{2i-1}\}$, and define $X^{4}_{i}$ ($X^{5}_{i}$)
as the set of vertices in $V(G) - V(C)$ of degree 4 (resp.\ degree 5)
that have a neighbor in $S_{i}$.
Let $j$ be the largest index such that
the following three properties hold:
\begin{itemize}
\item $G-S_{j}$ is $3$-connected;
\item $|X^{4}_{j}|=j$, and
\item $|X^{5}_{j}|=2j$.
\end{itemize}
Note that $|X^{4}_{1}|=1$, $|X^{5}_{1}|=2$, and $G-S_{1}$ is $3$-connected by
Lemma~\ref{lemma-3-cutsets}; thus, $j$ is well defined.
We distinguish two cases, depending on whether $j=k$ or $j < k$.

{\bf Case 1.} $j=k$: We have $s(G - S) \geq s(G) - k$. Since
$|X^{4}_{k}|=k$ and $|X^{5}_{j}|=2k$, no two vertices in $S$ ($=S_{k}$) have a common neighbor
in $V(G) - V(C)$. Thus, every vertex in $V(G) - V(C)$ has at most one neighbor
in $S$. Also, every vertex in $V(C)-S$ has exactly two neighbors in $S$, because
$C$ is almost induced. Since $G-S$ is connected, 
\begin{align*}
\phi(G) & = \phi(G - S) +
k\phi(5) + k\big(\phi(5) - \phi(3)\big) + k\big(\phi(4) - \phi(3)\big) + 2k\big(\phi(5) - \phi(4)\big)\\
& = \phi(G - S) + k\frac{16}{3}
 \geq \frac{16}{3}s(G - S) - 1 + k\frac{16}{3}
\geq \frac{16}{3}s(G) - 1,
\end{align*}
contradicting the fact that $G$ is a counter-example.

{\bf Case 2.} $j < k$: Here, we consider the set $S_{j+1}$.
Exactly $j$ vertices in $V(C)-S_{j+1}$ have two neighbors
in $S_{j+1}$, and exactly two have one. Let $G' :=G - S_{j+1}$.
For $\ell \in \{4,5\}$, let
$$
\Delta^{\ell}_{j+1} := \sum_{v\in X^{\l}_{j+1}}\big(\phi_{G}(v) - \phi_{G'}(v)\big).
$$

We have
\begin{align*}
\phi(G) & = \phi(G') +
(j+1)\phi(5) + j\big(\phi(5) - \phi(3)\big) + 2\big(\phi(5) - \phi(4)\big)
+ \Delta^{4}_{j+1} + \Delta^{5}_{j+1}\\
& = \phi(G') + (j+1)\frac{16}{3} - \frac16
+ \left(\Delta^{4}_{j+1} - (j+1)\frac23\right)
+ \left(\Delta^{5}_{j+1} - 2(j+1)\frac12\right).
\end{align*}

Observe that, by our choice of $j$, every vertex in $X^{4}_{j+1}$ and
$X^{5}_{j+1}$ has at most two neighbors in $S_{j+1}$.
Since
$$
\phi(4) - \phi(2) \geq 2\big(\phi(4) - \phi(3)\big) + \frac16
$$
and
$$
\phi(5) - \phi(3) = 2\big(\phi(5) - \phi(4)\big) + \frac16,
$$
we have
$$
\Delta^{4}_{j+1} \geq \left\{
\begin{array}{lll}
(j+1)\frac23 + \frac16 & & \textrm{if } |X^{4}_{j+1}| < j +1\\[1ex]
(j+1)\frac23 & & \textrm{otherwise},
\end{array}
\right.
$$
and
$$
\Delta^{5}_{j+1} \geq \left\{
\begin{array}{lll}
2(j+1)\frac12 + \frac16 & & \textrm{if } |X^{5}_{j+1}| < 2(j +1),\\[1ex]
2(j+1)\frac12 & & \textrm{otherwise}.
\end{array}
\right.
$$
Thus, if $|X^{4}_{j+1}| < j +1$ or $|X^{5}_{j+1}| < 2(j +1)$,
since $s(G') \geq s(G) - (j+1)$ and that $G'$ is connected, 
\begin{align*}
\phi(G) \geq \phi(G') + (j+1)\frac{16}{3}
\geq \frac{16}{3}s(G') - 1 + (j+1)\frac{16}{3}
\geq \frac{16}{3}s(G) - 1,
\end{align*}
a contradiction.
Therefore, $|X^{4}_{j+1}| = j +1$ and $|X^{5}_{j+1}| = 2(j +1)$.
By definition of $j$, this implies that $G'$ is $2$-connected but not $3$-connected.
Note also that $\Delta^4_{j+1} = (j+1)\frac23$ and
$\Delta^5_{j+1} = 2(j+1)\frac12$, implying
$$
\phi(G) = \phi(G') + (j+1)\frac{16}{3} - \frac16.
$$

Every vertex of $C$ has degree 5 in $G$
and every vertex in $V(G) - V(C)$ has at most one neighbor
in $S_{j+1}$. Moreover, no degree-3 vertex in $V(G)-V(C)$ has a neighbor in $C$.
It follows that $G'$ has minimum degree $3$.

Suppose $s(G') = s(G' - e)$ for some edge $e\in E(G')$. Then
$\phi(G') \geq \phi(G' - e) + 1$ (since $G'$ has minimum degree at least $3$),
and since $G'-e$ is connected, 
\begin{align*}
\phi(G) = \phi(G') + (j+1)\frac{16}{3} - \frac16
\geq \phi(G' - e) + (j+1)\frac{16}{3} - \frac16 + 1
\geq \frac{16}{3}s(G) - 1,
\end{align*}
a contradiction. Hence, the graph $G'$ is critical,
which in turn implies that $G'$ is reduced.

Now, since $G'$ is reduced but not $3$-connected, we may apply
Theorem~\ref{prop-phi-2-fvs-connected} (\ref{prop-phi-2-fvs-connected-2nd-part})
on $G'$, yielding
\begin{align*}
\phi(G) \geq \phi(G') + (j+1)\frac{16}{3} - \frac16
\geq \frac{16}{3}s(G') + (j+1)\frac{16}{3} - \frac16
\geq \frac{16}{3}s(G) - 1,
\end{align*}
which is again a contradiction.
\end{proof}

Note that the last paragraph of the above proof relies crucially on the stronger hypothesis
in the reduced but not $3$-connected case.

Lemma~\ref{lemma-no-deg-5} concludes the heart of the proof, namely
showing that there is no vertex of degree $5$ in $G$. Now,
it only remains to deal with vertices of degree $3$ and $4$, which is fairly easy in comparison.

\begin{lemma}
$G$ is $4$-regular and has no subgraph isomorphic to $K_4$.
\end{lemma}
\begin{proof}
First, suppose that $G$ contains some vertex $v$ of degree 4 having a degree-3 neighbor.
Since $G-v$ is connected, 
\begin{align*}
\phi(G) & \geq \phi(G- v) + \phi(4) + 3\big(\phi(4) - \phi(3)\big) + \big(\phi(3) - \phi(2)\big)\\
& = \phi(G- v) + \frac{16}{3}
\geq \frac{16}{3}s(G) - 1,
\end{align*}
a contradiction.
Thus, $G$ is either cubic ($3$-regular) or $4$-regular. If $G$ is cubic, then, letting
$v\in V(G)$ be an arbitrary vertex of $G$, 
$$
\phi(G) \geq \phi(G- v) + \phi(3) + 3\big(\phi(3) - \phi(2)\big)
= \phi(G- v) + \frac{16}{3}
\geq \frac{16}{3}s(G) - 1,
$$
again a contradiction.
Hence, $G$ is $4$-regular.
Also, $\verts{G} \geq 6$, since $K_{5}$ is not a counter-example.

Now, assume $X \subset V(G)$ induces a subgraph of $G$ isomorphic to $K_{4}$. Let $u \in X$ and
let $v$ be the unique neighbor of $u$ in $V(G)-X$. Since $G \not \cong K_{5}$, there
is some vertex $w \in X$ that is not adjacent to $v$.

The two neighbors $x$ and $y$ of $u$ in $G - \{v,w\}$ are adjacent, hence
$$
s(G - \{u,v,w\}) = s(G - \{v,w\}) \geq s(G) - 2.
$$
Let $z$ be the unique neighbor of $w$ in $V(G) - (X \cup \{v\})$.
The vertices $x$ and $y$ have degree at most $2$ in $G - \{u,v,w\}$, and
$z$ has degree at most $3$ in that graph.
Combining these observations with the fact that $G - \{u,v,w\}$ is connected (since $G - \{v,w\}$ is), we deduce
\begin{align*}
\phi(G) & \geq \phi(G - \{u,v,w\}) +
3\phi(4) + 2\big(\phi(4) - \phi(2)\big) + \big(\phi(4) - \phi(3)\big) \\
& =\phi(G - \{u,v,w\}) + 2\frac{16}{3}
\geq \frac{16}{3}s(G) - 1,
\end{align*}
a contradiction.
Therefore, $G$ contains no subgraph isomorphic to $K_4$.
\end{proof}

We are now in a position to complete the proof of Theorem~\ref{prop-phi-2-fvs-connected}
by showing that $\phi(G) \geq \frac{16}{3} s(G) -1$, and thus that
$G$ is not a counter-example to
part (\ref{prop-phi-2-fvs-connected-1st-part}) of Theorem~\ref{prop-phi-2-fvs-connected}, a final contradiction.

\begin{lemma}
$\phi(G) \geq \frac{16}{3} s(G) -1$.
\end{lemma}
\begin{proof}
The proof is similar to that of Lemma~\ref{lemma-no-deg-5}. The main difference is that, in the latter proof,
a vertex $v$ of the cycle $C$ with two neighbors in the stable set $S_i$ has degree $3$ in $G-S_i$, while here
it will have degree $2$ in $G-S_i$. In some cases, we will need to eliminate these degree-$2$ vertices
using the relevant operations (cf.\ Lemma~\ref{lemma-2fvs}).
A second difference is that here we are able to choose $C$ simply as a {\em shortest} even cycle in $G$,
which will ensure that no two vertices in $S$ have a common neighbor outside $V(C)$ when $|C|\geq 6$, simplifying somewhat the case analysis. (We could not have done that in the proof of Lemma~\ref{lemma-no-deg-5} because the graph might have contained $K_4$ as a subgraph.)

Let  $C$ be a shortest even cycle in $G$. Since $G$ is $4$-regular and has no subgraph isomorphic to $K_4$,
by Lemma~\ref{lem-cycle} such a cycle exists, and it is almost induced.
Enumerate the vertices of
$C$ in order as $v_{1}, v_{2}, \ldots, v_{2k}$ so that $S:=\{v_{1}, v_{3}, \dots, v_{2k-1}\}$
is a stable set of $G$. We may further assume that if $C$ is not induced, then the unique chord
of $C$ is incident to $v_{2k}$.

First suppose that $|C|=4$. Then $G-S$ is connected. We have $s(G-S) \geq s(G) -2$, and
$$
\phi(G) \geq \phi(G-S) + 4\phi(4) + 4(\phi(4) - \phi(3))
= \phi(G-S) + 10 + \frac23
\geq \frac{16}{3} s(G) -1.
$$

Next, assume $|C| \geq 6$.
Here we cannot simply remove $S$ from $G$, since $G-S$ might no longer be connected.
For each $i\in \{1, 2, \dots, k\}$,
let $S_{i}:=\{v_{1}, v_{3}, \dots, v_{2i-1}\}$,
let $x_i$ a neighbor of $v_{2i}$ outside $V(C)$, and let
$M_i := \{v_2x_1, v_4x_2, \dots, v_{2i}x_i\}$. Observe that no two vertices of $C$ which are at even distance on $C$
share a common neighbor outside $C$, for otherwise there would be an even cycle shorter than $C$.
(Here we use that $|C| \geq 6$.)
Thus, no two vertices in $S$ have a common neighbor outside $V(C)$, and
no two vertices in $V(C)-S$ have a common neighbor outside $V(C)$.
In particular, $M_i$ is a matching of $G$ for each $i\in \{1, 2, \dots, k\}$.

Let $j$ be the largest index in $\{1, 2, \dots, k\}$ such that
\begin{itemize}
\item $G-S_{j}$ is $2$-connected, and
\item none of $v_2, v_4, \dots, v_{2j-2}$ lies in a triangle in $G-S_j$.
\end{itemize}
Since $G-S_{1}$ is $2$-connected and the second condition is vacuous for $j=1$, the two properties hold for $j=1$,
and thus the index $j$ is well defined.
We distinguish two cases, depending on whether $j=k$ or $j < k$.

{\bf Case 1.} $j=k$: We have $s(G - S) \geq s(G) - k$ and
\begin{align*}
\phi(G) & = \phi(G - S) +
k\phi(4) + k\big(\phi(4) - \phi(2)\big) + 2k\big(\phi(4) - \phi(3)\big)\\
& = \phi(G - S) + k\frac{16}{3}
\geq \frac{16}{3}s(G - S) - 1 + k\frac{16}{3}
\geq \frac{16}{3}s(G) - 1,
\end{align*}
as desired.

{\bf Case 2.} $j<k$: Here we know that $G- S_{j+1}$ is connected (but perhaps not $2$-connected).

First suppose that $v_{2j}$ is in a triangle of $G- S_{j+1}$, and let $G' := G - (S_{j+1} \cup \{v_{2j}\})$.
Then $s(G') = s(G - S_{j+1}) \geq s(G) - (j+1)$. Also, $x_j$ has degree at least $3$ in
$G- S_{j+1}$, since each vertex of $V(G) -V(C)$ sees at most one vertex from $S$ in $G$.
(This is also true for the other neighbor of $v_{2j}$ in $G- S_{j+1}$ {\em if} it is outside $V(C)$; note however that this second neighbor
could be the vertex $v_{2k}$ in case $C$ has a chord.)
Since $\phi(3) - \phi(2) \geq \phi(4) - \phi(3)$, considering the vertex $x_j$ we deduce
that $\phi(G - S_{j+1}) \geq \phi(G') + (\phi(4) - \phi(3)) = \phi(G') + \frac23$.

Combining the previous observations, 
\begin{align*}
\phi(G) & = \phi(G - S_{j+1}) + (j+1)\phi(4) + j(\phi(4) - \phi(2))
+ 2(j+1)(\phi(4) - \phi(3)) + 2(\phi(4) - \phi(3))\\
& = \phi(G - S_{j+1}) + (j+1)\frac{16}{3} - \frac23\\
& \geq \phi(G') + (j+1)\frac{16}{3}
\geq \frac{16}{3}(s(G) - (j+1)) - 1 + (j+1)\frac{16}{3}
= \frac{16}{3}(s(G) - 1).
\end{align*}
Hence we may assume that $v_{2j}$ is not in a triangle in $G- S_{j+1}$, and it follows that the latter graph
is connected but not $2$-connected.

Let $z$ be a cutvertex of $G- S_{j+1}$. We may assume that $z$ has been chosen so that it is distinct from
$v_2, v_4, \dots, v_{2j}$ (if not, simply replace $z$ by one of its two neighbors).
Now we will focus on the graph $G- S_{j}$. Clearly, $\{v_{2j+1}, z\}$ is a $2$-cutset of that graph.
More importantly, $\{v_{2j+1}, z\}$ is also a $2$-cutset of $G' := (G - S_j) / M_j$, the graph obtained from
$G - S_j$ by contracting each edge of the matching $M_j$. Thus $G'$ is not $3$-connected.
On the other hand, $G'$ is $2$-connected, since any cutvertex of $G'$ would also be a cutvertex of
$G - S_j$ (which is $2$-connected).

We claim that $v_2, v_4, \dots, v_{2j}$ are the only vertices of degree $2$ in $G- S_{j}$, and that every other vertex has degree
at least $3$.  This is clear if none of $v_2, v_4, \dots, v_{2j}$ is incident to a chord of $C$, since no
vertex from $V(G)-V(C)$ sees two vertices from $S$ in $G$.
If, on the other hand, there is a chord of $C$ incident to one of these vertices, then it is of the form
$v_{2\ell}v_{2k}$ for some $\ell \in \{1, 2, \dots, j\}$, and we observe that $v_{2k}$ has degree $3$ in
$G- S_{j}$ since $j < k$. Again this shows that $v_2, v_4, \dots, v_{2j}$ are the only vertices of degree $2$ in $G- S_{j}$.

By the previous observation, it follows that $G'$ has minimum degree $3$.
If some edge $e$ of $G'$ is not critical then, since $G' - e$ is connected and $s(G'-e) = s(G') = s(G- S_j)$,
\begin{align*}
\phi(G) & = \phi(G - S_{j}) + j\phi(4) + (j-1)(\phi(4) - \phi(2))
  + 2j(\phi(4) - \phi(3)) + 2(\phi(4) - \phi(3))\\
& = \phi(G - S_{j}) + j\frac{16}{3} - \frac23\\
& = \phi(G') + j\frac{16}{3} - \frac23\\
& \geq \phi(G' - e) + j\frac{16}{3}
\geq \frac{16}{3}(s(G) - j) - 1 + j\frac{16}{3}
= \frac{16}{3}(s(G) - 1).
\end{align*}
Thus we may assume that $G'$ is critical, implying that $G'$ is reduced. Since $G'$ is not $3$-connected, applying
part (\ref{prop-phi-2-fvs-connected-2nd-part}) of Theorem~\ref{prop-phi-2-fvs-connected} on $G'$ yields
\begin{align*}
\phi(G) & = \phi(G - S_{j}) + j\phi(4) + (j-1)(\phi(4) - \phi(2))
+ 2j(\phi(4) - \phi(3)) + 2(\phi(4) - \phi(3))\\
& = \phi(G - S_{j}) + j\frac{16}{3} - \frac23\\
& = \phi(G') + j\frac{16}{3} - \frac23\\
& \geq \frac{16}{3}s(G') + j\frac{16}{3}  - \frac23
\geq \frac{16}{3}(s(G) - j) + j\frac{16}{3}  - \frac23
\geq \frac{16}{3}(s(G) - 1),
\end{align*}
which concludes the proof.
\end{proof}

\bibliographystyle{abbrv}
\bibliography{bibliography}

\end{document}